\documentclass{daj}

\usepackage[latin9]{inputenc}
\usepackage{amsmath}
\usepackage{amsthm}
\usepackage{amssymb}
\usepackage{setspace}

\numberwithin{equation}{section}
\usepackage{graphicx}
\usepackage{cite}
\usepackage{latexsym}
\usepackage{amscd}
\usepackage{tikz}
\usepackage{mathrsfs}
\usepackage{url}
\usepackage[english]{babel}
\usepackage{amsfonts}
\usepackage{mathtools}
\usepackage{thm-restate}

\numberwithin{equation}{section}
\def\R{\mathbb R}
\def\Z{\mathbb Z}

\def\P{\mathcal P}
\def\N{\mathbb N}
\def\CN{\mathcal N}
\def\E{\mathbb E}
\def\L{\log}
\def\LL{\log\log}
\def\LLL{\log\log\log}
\def\LLLL{\log\log\log\log}

\def\A{\mathcal A}
\def\B{\mathcal B}
\def\X{\mathcal X}
\def\Y{\mathcal Y}
\def\Q{\mathbb Q}

\def\ee{\varepsilon}

\def\gcd{\operatorname{gcd}}

\DeclareMathOperator\vol{Vol}
\DeclarePairedDelimiter\ceil{\lceil}{\rceil}
\DeclarePairedDelimiter\floor{\lfloor}{\rfloor}
\newcommand{\Mod}[1]{\;(\mathrm{mod}\;#1)}

\newtheorem{theorem}{Theorem}[section]
\newtheorem{lemma}[theorem]{Lemma}
\newtheorem{proposition}[theorem]{Proposition}

\newtheorem{corollary}[theorem]{Corollary}
\newtheorem{conjecture}[theorem]{Conjecture}

\theoremstyle{remark}
\newtheorem*{remark}{Remark}

\theoremstyle{definition}
\newtheorem{definition}[theorem]{Definition}

\theoremstyle{remark}

\numberwithin{equation}{section}

\newcommand{\upperz}{2}
\newcommand{\Lpower}{2}

\dajAUTHORdetails{%
  title = {On Product Sets of Arithmetic Progressions}, 
  author = {Max Wenqiang Xu and Yunkun Zhou},
  plaintextauthor = {Max Wenqiang Xu, Yunkun Zhou},
}   

\dajEDITORdetails{%
   year={2023},
   number={10},
   received={4 January 2022},   
   published={25 July 2023},  
   doi={10.19086/da.84267},       
}   

\begin{document}
\begin{frontmatter}[classification=text]

\title{On Product Sets of Arithmetic Progressions} 


\author[maxxu]{Max Wenqiang Xu\thanks{Supported by the Cuthbert C. Hurd Graduate Fellowship in Mathematics, Stanford}}
\author[yunkunzhou]{Yunkun Zhou\thanks{Supported by NSF GRFP Grant DGE-1656518}}

 	\begin{abstract}
We prove that the size of the product set of any finite arithmetic progression $\mathcal{A}\subset \mathbb{Z}$ satisfies
	\[|\mathcal A \cdot \mathcal A| \ge \frac{|\mathcal A|^2}{(\log |\mathcal A|)^{2\theta +o(1)} } ,\]
where $2\theta=1-(1+\log\log 2)/(\log 2)$ is the constant appearing in the celebrated Erd\H{o}s multiplication table problem. This confirms a conjecture of Elekes and Ruzsa from about two decades ago.
		
If instead $\mathcal{A}$ is relaxed to be a subset of a finite arithmetic progression in integers with positive constant density, we prove that
\[|\mathcal A \cdot \mathcal A | \ge \frac{|\mathcal A|^{2}}{(\log |\mathcal A|)^{2\log 2- 1 + o(1)}}. \]
This solves the typical case of another conjecture of Elekes and Ruzsa on the size of the product set of a set $\mathcal{A}$ whose sumset is of size  $O(|\mathcal{A}|)$. 

Our bounds are sharp up to the $o(1)$ term in the exponents. We further prove asymmetric extensions of the above results. 
	\end{abstract}
\end{frontmatter}

\section{Introduction}

	The celebrated Erd\H{o}s multiplication table problem asks to estimate how many numbers can be represented as the product of two positive integers which are at most $N$. It is a classical result \cite{Erdos55} that there are $o(N^{2})$ such numbers, which can be proved by simply considering the typical number of prime factors of an integer $n\le N$. Erd\H{o}s \cite{Erd60} first determined the answer up to a $(\log N)^{o(1)}$ factor:  
	\begin{equation}\label{eqn: delta}
	    \Big|[N] \cdot [N]\Big| = \frac{N^{2}}{(\log N)^{2\theta+o(1)}},	\end{equation} 
	where $[N]=\{1,\ldots,N\}$, $\theta = 1- \frac{1+\log \log 4}{\log 4}$ (which implies that $2\theta=1-\frac{1+\log\log 2}{\log 2}$), and $\A\cdot \A$ is defined to be $\{aa': a, a'\in \A\}$ for any set $
	    \A$. 
	The finer-order term was improved later by Tenenbaum \cite{Tenenbaum}. Remarkably, Ford \cite{Ford08} has determined the quantity up to a constant factor:
	\begin{equation}\label{eqn: ford}
	     \Big|[N] \cdot [N]\Big| \asymp \frac{N^{2}}{(\log N)^{2\theta}(\log \log N)^{3/2}}. 
	\end{equation}
	There is also work on generalizations of the Erd\H{o}s multiplication table problem to higher dimensions: see \cite{ Kou14} for references.

In this paper, we are interested the same problem but with $[N]$ replaced by an arbitrary arithmetic progression $\A$ of $N$ integers. In some cases the size of the product set can be as large as the trivial upper bound $\binom{N+1}{2}$, which is of order $N^2$. One example that achieves this bound is $\{1+kd:k\in [N]\}$ where $d > 2N$. It is interesting to know whether the size of the product set of any arithmetic progression with given length $N$ can be substantially smaller than the bound in \eqref{eqn: delta}. Elekes and Ruzsa \cite{ElekesRuzsa} conjectured that it cannot. In this paper, we prove their conjecture.


\begin{theorem}\label{cor:  AA 86}
Let $\A\subseteq\Z$ be a finite arithmetic progression and $\theta = 1- \frac{1+\log \log 4}{\log 4}$. Then 
\[|\A \cdot \A| \geq  |\mathcal A|^2(\log |\mathcal A|)^{-2\theta- o(1)}  . \]
\end{theorem}
The strongest lower bound we can prove is $|\mathcal A|^2(\log |\mathcal A|)^{-2\theta} (\log \log |\A|)^{-7-o(1)}$ (Theorem~\ref{thm: best}), which is sharp up to a power of $\log \log |A|$.

We also prove an asymmetric version of the conjecture, which strengthens Theorem~\ref{cor:  AA 86}.

\begin{theorem}\label{cor:  AB 86}
Let $\A, \B \subseteq\Z$ be two finite arithmetic progressions with lengths $2 \leq |\B|\le |\A|$ and $\theta = 1- \frac{1+\log \log 4}{\log 4}$. Then
\[|\A \cdot \B| \geq  {|\A||\B|}{(\log |\A|)^{-\theta - o(1)}(\log |\B|)^{-\theta}     } , \]
where the $o(1)$ terms go to zero as $|\A|$ tends to infinity.
\end{theorem}

\begin{remark}
The two theorems above are sharp up to a $(\log |\A|)^{o(1)}$ factor. The proof we present gives the error term in the form $\exp(O(\sqrt{\log \log |\A|}\LLL |\A|))$. In fact, a stronger bound up to a $(\log \log |\A|)^{O(1)}$ factor can be obtained. We leave details of the proof to the stronger error term to Appendix \ref{sec:Smirnov-stronger}. 
\end{remark}


A particular structural property of arithmetic progressions $\A$  is that it has small sumset. The celebrated sum-product conjecture \cite{ES83} states that for any $\A \subseteq \Z$, 
\begin{equation}\label{eqn: sum product}
    \max\{|\A+\A|, |\A\cdot \A|\} \gg   |\A|^{2-o(1)}, 
\end{equation}
where $\A+\A = \{a+a': a, a' \in \A\}$ and $\ll$ is to denote up to a multiplicative constant.
There has been much progress towards this conjecture (see \cite{RS20} for the current record), though it is still wide open. One extremal case of the sum-product conjecture is the case where one of the sumset and product set is very small, in which, the conjecture is known to be true. Indeed, 
Chang \cite{Chang} proved that 
\[|\A \cdot \A|\ll |\A| \implies  |\A+\A|\gg  |\A|^{2}, \]
which is sharp up to a multiplicative constant factor. See \cite{MRSS, RSS2020} for discussions in other settings. 

We are interested in the other extremal case where $|\A + \A|/|\A|$ is bounded, which seems to be much more challenging. In contrast to Chang's result \cite{Chang}, even if the doubling number (i.e. $|\A+\A|/|\A|$) is at most $2$, the product set can still be smaller than $|\A|^{2}$ by a polylogarithmic factor (see \eqref{eqn: ford}). 
Nathanson and Tenenbaum proved \cite{NT99} that if $|\A+\A|<3|\A|-4$, then $|\A\cdot \A|\gg |\A|^{2}/(\log |\A|)^{2}$. 
This result was later generalized by Chang \cite{Chang-unpub} to $h$-fold product sets. Finally, Elekes and Ruzsa \cite{ElekesRuzsa} proposed the following conjecture.  

\begin{conjecture}[Elekes and Ruzsa \cite{ElekesRuzsa}]\label{conj: ER}
Let $\A \subseteq\Z$ be a finite set. Then the following holds. 
\[ |\A +\A| \ll |\A| \implies |\A \cdot \A| \ge  |\A|^2(\log |\A|)^{1-2\log 2 - o(1)} . \]
\end{conjecture} Elekes and Ruzsa \cite{ElekesRuzsa} showed that if $|\A+\A|\ll |\A|$, then $|\A\cdot \A| \gg |\A|^{2}(\log |\A|)^{-1}$ when $\A$ is a finite set of integers.
As a corollary of Solymosi's remarkable result \cite{Solymosi09}, the same implication holds when $\A$ is a finite set of reals. We remark that their proofs did not use specific properties of the integers, and we improve their bounds by taking the arithmetic information of integers into account. Indeed, the $2\log 2-1$ exponent has a natural arithmetic interpretation. By the Sathe-Selberg formula \cite{Selberg}, the number of $n\le N$ with $\Omega(n) = (2+o(1))\log \log N$ is about $N/(\log N)^{2\log 2 -1 + o(1)}$, where $\Omega(n)$ is the number of prime factors of $n$. Further we claim that the bound in Conjecture~\ref{conj: ER} is optimal up to $o(1)$ factor in the exponent. One simple example is to take $\A \subset [N]$ containing set of all numbers $n$ with $\Omega(n) = (1+o(1))\log \log N$. By the Hardy-Ramanujan theorem one has $|\A|= (1+o(1))N$. Since $A \cdot A$ is contained in the set of integers with $\Omega(n) = (2+o(1))\log \log N$, combining this with the previously mentioned application of Sathe-Selberg's formula, the claim follows.

We believe that to get the conjectured exponent $2\log 2 -1$, one has to take the arithmetic information into account. We also remark that for similar arithmetic reasons, the constant $2\log 2-1$ appears in recent work \cite{dMT, Ma, SoundXu}.

By the Freiman-Ruzsa theorem \cite{Freiman73, Ruzsa94}, every set with constant doubling number is a dense subset of a generalized arithmetic progression of constant dimension. A natural case is that $\A$ is a dense subset of an arithmetic progression. In fact, in some sense this is also the typical case \cite{camp20, campos2021typical}. 
In this paper, we prove that the Elekes-Ruzsa conjecture is true in this case.

\begin{theorem}\label{cor: AA2log2-1}
Let $\delta \in (0,1]$ and  $\A$ be a subset of an arithmetic progression $\P \subseteq\Z$ with $|\A| \geq \delta |\P|$, then
\[|\A \cdot \A| \gg_{\delta} |\A|^{2}(\log |\A|)^{1-2\log 2 - o(1)}.\]
\end{theorem}
We remark that Pomerance and S\'ark\"{o}zy \cite{PS90} proved a special case of Theorem~\ref{cor: AA2log2-1} $(\P= [N])$.
We also prove a more general asymmetric version. 
\begin{theorem}\label{cor: AB2log2-1}
Let $\delta \in (0,1]$ and $\A, \B $ be two subsets of two finite arithmetic progressions $\P_1, \P_2 \subseteq\Z$, and $|\A| \ge \delta |\P_1|$, $|\B| \ge \delta |\P_2|$, then
\[|\A \cdot \B| \gg_{\delta} |\A||\B|(\log |\A|)^{\frac{1}{2}-\log 2 - o(1)} (\log |\B|)^{\frac{1}{2}-\log 2 - o(1)}.\]
\end{theorem}

Although Theorem~\ref{cor: AA2log2-1} does not solve Conjecture~\ref{conj: ER} completely, it can be used to prove the case for doubling number up to $4$ (in terms of difference set). In this case, Eberhard, Green and Manners \cite[Theorem 6.4]{EGM14} proved that $\A$ must be a subset of an arithmetic progression with positive density, and thus the following corollary can be  deduced from Theorem~\ref{cor: AA2log2-1} immediately.
\begin{corollary}\label{cor: 2log2-1: doubling 4}
Let $\ee>0$. 
Let $\A \subseteq\Z$ be a finite set with 
$
|A-A|\le (4-\ee)|A|
$, then 
\[|\A \cdot \A| \gg_{ \ee} |\A|^{2} (\log |\A|)^{1-2\log2 -o(1)}.  \]
\end{corollary}




In Theorem~\ref{cor: AB2log2-1}, one can relax the condition that $\A, \B$ are dense subsets of arithmetic progressions $\P_1, \P_2$, to the condition that $\A$ intersects an arithmetic progression $\P_1$ (with sizes $|\P_1| = \Theta( |\A|)$) with positive density, and a similar condition holds for $\B$ and $\P_2$. Such $\A$ and $\B$ may have large sumsets. 

Our Theorem~\ref{cor:  AB2log2-1} also easily implies the following observation. Consider the restricted sumsets $$\A+_{G}\B: = \{a+b: (a,b)\in G)\}, \quad |G|\ge (1-\ee)|\A||\B|,$$
where $\ee$ is a small absolute constant. If the restricted sumset is small, then one can get similar lower bounds on $|\A \cdot \B|$ as in Theorem~\ref{cor: AB2log2-1}. The proof simply consists of using the structural results in \cite{Lev, SX19} to conclude that such set $\A, \B$ must overlap with some arithmetic progressions $\P_1, \P_2$ heavily, and then applying Theorem~\ref{cor:  AB2log2-1} to conclude the proof. We leave the details for the interested reader.

We prove Theorems~\ref{cor:  AB 86} and \ref{cor: AB2log2-1} by bounding the multiplicative energy from above.

\begin{definition}[Multiplicative energy]\label{def: energy}
Let $\A, \B$ be two finite subsets of integers. The multiplicative energy between $\A, \B$ is defined as 
\[ E_\times(\A, \B):= \left| \{(a_1, a_2, b_1, b_2)\in \A \times \A \times \B \times \B: a_1 b_1= a_2b_2\} \right| .\]
When $\A=\B$, we write $E_\times(\A):= E_\times(\A, \A)$, which is called the multiplicative energy of $\A$. 
\end{definition}
 
The first main theorem is about the multiplicative energy of any finite arithmetic progression. 
\begin{theorem}\label{thm: A energy 86} 
Let $\A \subseteq\Z$ be a finite arithmetic progression. Then there exists a subset $\A'\subseteq \A$ with size $|\A'|\ge |\A|(\L|\A|)^{-\theta-o(1)}$ such that $E_{\times } (\A') \ll |\A'|^{2}.$
\end{theorem}
A stronger version in terms of the $o(1)$ term is achieved in Appendix \ref{sec:Smirnov-stronger}. The second main theorem is about the multiplicative energy of a dense subset of any finite arithmetic progression. 
\begin{theorem}\label{thm: A energy 2log2-1}
Let $\delta \in (0,1]$ and let $\A $ be a subset of a finite arithmetic progression $\P\subseteq\Z$ such that $|\A| \ge \delta |\P| $. Then there exists a subset $\A'\subseteq \A$ with size $|\A'|\gg_\delta |\A|$ such that
\[ E_{\times}(\A') \ll_{\delta} |\A|^{2}(\log |\A|)^{2\log 2 -1+o(1)}. \]
\end{theorem}



\subsection{Proof ideas}
Our main focus will be on Theorem~\ref{thm: A energy 86} and Theorem~\ref{thm: A energy 2log2-1}, and all the rest of our results will follow from them fairly straightforwardly (See Section~\ref{Sec: deduction to energy}). 

Theorem~\ref{thm: A energy 86} and Theorem~\ref{thm: A energy 2log2-1} are well understood when the arithmetic progression is just the first $N$ integers (See \cite{Ten84,Erd60,ford2008sharp} for the special case of Theorem~\ref{thm: A energy 86} and  \cite{PS90} for the special case of Theorem~\ref{thm: A energy 2log2-1}.) The classical proofs of both special cases involve estimates on $\pi_k(x)$ (the number of integers up to $x$ with $k$ distinct prime factors) and tools like Shiu's Theorem on bounding mean values of multiplicative functions. The original arguments cannot be directly generalized because these tools are not applicable to general arithmetic progressions. Our main novelty is the deduction of the general case in the case where the parameters of the arithmetic progression lie in a certain range such that all classical tools are applicable. The deduction is completed by proving the desired upper bounds for multiplicative energies in the other cases via elementary counting arguments. After the deduction steps, we carefully employ and adapt the method developed in \cite{ford2008sharp, Erd60, Ford08, FordExtremal18} to conclude the proof of both Theorem~\ref{thm: A energy 86} and Theorem~\ref{thm: A energy 2log2-1}.



\subsection{Organization}
We organize the paper as follows. 
In Section~\ref{Sec: deduction to energy} we prove Theorems~\ref{cor:  AA 86}, \ref{cor:  AB 86}, \ref{cor: AA2log2-1} and \ref{cor: AB2log2-1} assuming Theorems~\ref{thm: A energy 86} and \ref{thm: A energy 2log2-1}. This is done by applying Cauchy-Schwarz type arguments (see Lemma~\ref{lem: CS}).
The remaining sections are devoted to proving the main Theorems~\ref{thm: A energy 86} and \ref{thm: A energy 2log2-1}. To prove Theorem~\ref{thm: A energy 2log2-1}, we first do a sequence of reduction steps in Section~\ref{sec:reduction} as a preparation to proving the essential case of Theorem~\ref{thm: A energy 2log2-1} 
in Section~\ref{sec:square-free}. 
We study the number of elements with a fixed number of prime factors in arithmetic progressions in Section~\ref{sec: Smirnov} and crucially use it to complete the proof of Theorem~\ref{thm: A energy 86} in Section \ref{Sec: main thm}. Finally in Section~\ref{sec:concluding} we make conjectures about both problems.

\subsection{Notations} For two functions $f, g: \R\to \R$, we write $f\ll g, g\gg f, g= \Omega(f)$ or $f = O(g)$ if there exists a positive constant $C$ such that $f\le C g$, and we write $f \asymp g$ or $f=\Theta(g)$ if $f\ll g$ and $g \gg f$. We write
$f=o(g)$ if $f(x) \le \epsilon g(x)$ for any $\epsilon>0$ when $g(x)$ is sufficiently large. We write $\ll_\delta$ or $\gg_{\delta}$ if the implicit constant depends on $\delta$.
Throughout the paper, $\theta$ always denotes the constant $1- \frac{1+\log \log 4}{\log 4}\approx 0.043...$. For two sets $\A$ and $\B$, we write $\A\B = \A\cdot \B = \{ab:a\in \A, b\in \B\}$. When $\A = \{a\}$ is a singleton, we write $a\B = \{a\}\cdot\B$. For positive integer $n$, $\omega(n)$ denotes the number of distinct prime factors of $n$ and $\phi(n)$ is the Euler's totient function. All logarithms are base $e$.


\section{Reduction to multiplicative energy estimates}\label{Sec: deduction to energy}
In this section, we show that Theorems~\ref{thm: A energy 86} and \ref{thm: A energy 2log2-1} imply all the other main results. 
The implications follow from the Cauchy-Schwarz inequality (see Lemma~\ref{lem: CS}). The proof of Lemma~\ref{lem: CS} is standard (e.g., see \cite{TaoVu} Corollary 2.10 for an additive version): we include it here for completeness.

\begin{lemma}\label{lem: CS}
Let $\A, \B$ be nonempty sets of nonzero integers. Then we have
\begin{enumerate}
    \item\label{item:CS-1} $|\A\cdot \B|  \ge \frac{|\A|^{2}|\B|^2}{E_\times (\A, \B)}.$
    \item\label{item:CS-2} $ E_{\times}(\A,\B) \le \sqrt{E_{\times}  (\A) E_{\times} (\B) }$.
\end{enumerate}
\end{lemma}

\begin{proof}\
For any two sets of integers $\mathcal{X}, \mathcal{Y} \subset \Z\backslash \{0\} $, we define
\[ r_{\X \cdot \Y}(m): = | (x, y)\in \X \times \Y: xy = m  |, \]
and similarly define define ratio representation functions $r_{\X/\Y} (m)$ for $m\in \Q$. 
By Definition~\ref{def: energy} and the Cauchy-Schwarz inequality, 
\[|\A \cdot \B | \cdot E_{\times } (\A, \B) = \sum_{m\in \A\cdot\B}1 \cdot  \sum_{m\in \A\cdot\B}r_{\A\cdot\B}(m)^{2} \ge  \left(\sum_{m\in \A\cdot\B} r_{\A\cdot\B}(m) \right)^{2} = |\A|^{2}|\B|^{2}.
\]
This gives \eqref{item:CS-1}. By Definition~\ref{def: energy}, 
and noting that $a_1a_2 =a_1'a_2'$ is equivalent to
$\frac{a_1}{a_1'} = \frac{a_2'}{a_2}$, we have
\[E_{\times}(\A)
=  \left|\{ (a_1, a_2; a_1', a_2')\in \A^{4}: a_1a_2 =a_1'a_2' \} \right| = \sum_{m\in \Q}r_{\A/\A}(m)^2.
\]
We have a symmetric equation for $\B$. By the Cauchy-Schwarz inequality, we have
\[ E_{\times}(\A) E_{\times}(\B) = \left(\sum_{m\in \Q}r_{\A/\A}(m)^{2}\right) \cdot 
\left(\sum_{m\in \Q}r_{\B/ \B}(m)^{2}\right) \ge \left(  \sum_{m\in \Q} r_{\A/\A}(m)r_{\B/ \B}(m) \right)^{2}.
\]
The right hand side is the square of the following counting
\[\left|\{ (a_1, a_2; b_1, b_2)\in \A \times \A \times \B \times \B:  a_1/a_2 = b_1/b_2  \} \right | = E_{\times}(\A, \B) . \]
The equality is because $a_1/a_2 = b_1/b_2$ is equivalent to $a_1b_2=a_2b_1$. Hence we have \eqref{item:CS-2}.
\end{proof}

Assuming Theorems~\ref{thm: A energy 86} and \ref{thm: A energy 2log2-1}, we prove all the theorems regarding product sets (Theorems~\ref{cor:  AA 86}, \ref{cor:  AB 86}, \ref{cor: AA2log2-1} and \ref{cor: AB2log2-1}) via Lemma~\ref{lem: CS}. Before we proceed to the proof, we remark that the condition in Lemma~\ref{lem: CS} that $\A, \B$ do not contain $0$ is a mild technical restriction, as deleting one single element changes the size of product sets by at most  $O(|\A| + |\B|)$, and changes multiplicative energy by at most $O(|\A||\B|)$, which are negligible for our purpose. 

\begin{proof}[Proof of Theorems~\ref{cor:  AA 86} and \ref{cor:  AB 86} assuming Theorem~\ref{thm: A energy 86}]
By Theorem~\ref{thm: A energy 86}, there exist subsets $\A'\subseteq \A$ and $\B'\subseteq \B$ with sizes $|\A'|\gg  |\A|(\L|\A|)^{-\theta-o(1)},  |\B'| \gg  |\B|(\L|\B|)^{-\theta-o(1)},$
such that $E_{\times}(\A')\ll |\A'|^2, E_{\times} (\B')\ll |\B'|^2.$
These energy bounds give
\[E_{\times} (\A', \B') \le \sqrt{E_{\times}(\A') E_{\times} (\B')} \ll |\A'||\B'|,   \]
where the first inequality is due to \eqref{item:CS-2} of Lemma~\ref{lem: CS}. By \eqref{item:CS-1} of Lemma~\ref{lem: CS} we conclude that
\[ |\A\cdot \B|\ge |\A'\cdot \B'|\ge \frac{|\A'|^{2}|\B'|^{2}}{E_{\times}(\A',\B')} \ge \frac{|\A||\B|}{(\log |\A|)^{\theta + o(1)}(\log |\B|)^{\theta + o(1)}    }.
\]
We thereby prove Theorem~\ref{cor:  AB 86}. Taking $\B = \A$ in Theorem~\ref{cor:  AB 86} gives Theorem~\ref{cor:  AA 86}. 
\end{proof}

\begin{proof}[Proof of Theorems~\ref{cor: AA2log2-1} and \ref{cor: AB2log2-1} assuming Theorem~\ref{thm: A energy 2log2-1}]
By applying Theorem~\ref{thm: A energy 2log2-1} to $\A$ and $\B$, we get subsets $\A'\subseteq \A$ and $\B'\subseteq \B$ with sizes $|\A'|\gg_{\delta} |\A|,  |\B'|\gg_{\delta} |\B|,$
and energy bounds
\[ E_{\times}(\A') \ll_{\delta} |\A|^{2}(\log |\A|)^{2\log 2 -1+o(1)},\quad 
E_{\times}(\B') \ll_{\delta} |\B|^{2}(\log |\B|)^{2\log 2 -1+o(1)}.
\]
By applying \eqref{item:CS-2} of Lemma~\ref{lem: CS} we get 
\[
E_{\times}(\A, \B) \ll_{\delta} |\A||\B| (\log |\A|)^{\log 2 -\frac{1}{2} +o(1)}(\log |\B|)^{\log 2 -\frac{1}{2} +o(1)} .
\]
Now Theorem~\ref{cor: AB2log2-1} follows by \eqref{item:CS-1} of Lemma~\ref{lem: CS}.
Taking $\B=\A$ in Theorem~\ref{cor: AB2log2-1} gives Theorem~\ref{cor: AA2log2-1}.  
\end{proof}

\section{Proof of Theorem~\ref{thm: A energy 2log2-1}: reduction to the essential case}\label{sec:reduction}
In this section, we want to reduce Theorem~\ref{thm: A energy 2log2-1} to the following case, which is proved in the Section~\ref{sec:square-free}.

\begin{restatable}{theorem}{reduced}\label{thm:reduced}
Let $\delta \in (0, 1]$. Let $\P = \{a+id:  0\leq i < L\}$ be an arithmetic progression with common difference $d$ and length $L$, with $L, d > 0$, $L\log L > a \geq dL$, and $\gcd(a, d) = 1$. If $\A\subseteq \P$ is a subset of size at least $\delta L$ containing only square-free elements and $L$ is sufficiently large with respect to $\delta$, then there exists a subset $\A'\subseteq \A$ with $|\A'|\gg |\A|$ such that
\begin{equation}\label{eqn:thm-reduced}
   E_\times(\A') \ll L^2(\log L)^{2\log2 -1+ o(1)}.
\end{equation}
\end{restatable}



We start by giving an upper bound on the multiplicative energy in some special cases. When $\gcd(a, d) = 1$, we have the following observation on the multiplicative energy.

\begin{lemma}\label{lem:large-a}
        Let $\P = \{a + id: 0\leq i < L\}$ be an arithmetic progression with $\gcd(a, d) = 1$ and $a > 0, d>0$. For any $\A\subseteq \P$, we have $E_\times(\A) \leq 2|\A|^2 + 4\frac{L^3}{a}(1+\L L).$
\end{lemma}
\begin{proof}
   The number of solutions $(a_1, a_2; a_3, a_4)\in \A^4$ to $a_1a_2 = a_3a_4$ with $\{a_1, a_2\} = \{a_3, a_4\}$ is at most $2|\A|^2$. We next estimate the number of off diagonal solutions.
    
We begin by parameterizing solutions $(a_1, a_2; a_3, a_4)$. Write $x_1 = \gcd(a_1, a_3)$ and $a_1 = x_1y_1$,  $a_3 = x_1y_2$. As a result, $a_2/y_2 = a_4/y_1$, which we denote as $x_2$. Therefore we derive a tuple $(x_1, x_2; y_1, y_2)$ such that $x_iy_j\in \A$ for all $i, j\in \{1, 2\}$. As $\{a_1, a_2\}\ne \{a_3, a_4\}$, one has
$x_1\ne x_2$ and $y_1\ne y_2$. By symmetry we may assume that $x_1 < x_2$ and $y_1 < y_2$ and thereby decrease the number of such tuples by a factor of $4$, i.e.   $E_\times(\A) - 2|\A|^2$ is at most $4|X|$ where
     \[X := \{(x_1, x_2; y_1, y_2): x_iy_j\in \A\;~~\forall \; i, j\in \{1, 2\}, 0 < x_1 < x_2, 0 < y_1 < y_2\}.\]

   Fix a choice of $x_1$. Since elements in $\A$ are included in the interval $[a, a+dL)$, one has that $\frac{x_2}{x_1} = \frac{x_2y_1}{x_1y_1} < 1+\frac{dL}{a}$. Together with $x_2 > x_1$, it implies $x_2\in (x_1, x_1+x_1dL/a)$. Since all elements in $\A$ are coprime to $d$, $x_i$ and $y_j$ are also coprime to $d$. As a result, $x_1y_1\equiv a \equiv x_2y_1 \pmod d$ implies that $x_1$ and $x_2$ are both congruent to $ay_1^{-1}$ mod $d$, and thus the number of choices for $x_2$ is at most $\lfloor\frac{x_1L}{a}\rfloor$. 
    
    Similarly,  $y_i\in [\frac{a}{x_1}, \frac{a}{x_1}+\frac{dL}{x_1})$ as $x_1y_i\in [a, a+dL)$. By the fact $y_1 < y_2$ and $ y_1 \equiv y_2 \mod d$, there is no valid choice for $(y_1, y_2)$ if $\frac{L}{x_1} \leq 1$. If $\frac{L}{x_1} > 1$, then the number of choices for $(y_1, y_2)$ is at most 
    \[\frac{1}{2}\left\lceil\frac{L}{x_1}\right\rceil\cdot \left(\left\lceil\frac{L}{x_1}\right\rceil - 1\right) \leq \frac{L^2}{x_1^2}.\]
    
    In summary, any such valid tuple satisfies $\frac{a}{L}\leq x_1 < L$. Once $x_1$ is fixed, the number of such tuples is at most $\frac{L^3}{ax_1}$. By summing over all choices of $x_1$, we conclude that \[E_\times(\A) \leq 2|\A|^2 + 4\sum_{x_1 < L}\frac{L^3}{ax_1} \leq 2|\A|^2 + \frac{4L^3}{a}(1+\L L).\]
\end{proof}

The above lemma gives the following corollary immediately. 
\begin{corollary}\label{cor:large-a}
        Let $\P = \{a + id: 0\leq i < L\}$ be an arithmetic progression with $\gcd(a, d) = 1$ and $a > 0, d>0$. Suppose $\A\subseteq \P$. If $a = \Omega(L\log L)$, then $E_\times(\A) = O(L^2)$.
\end{corollary}


Our next step is to reduce the problem to the case that the elements are all square free. 
\begin{lemma}[Square-free reduction]\label{lem: squarefree}
         Let $\P = \{a+id: 0\leq i < L\}$ be an arithmetic progression with $\gcd(a, d) = 1$ and $a, d > 0$. 
         Suppose $ a+dL < \frac{\delta^2}{9}L^2$ and $d\le L$. If $\A$ is an arithmetic progression with density $\delta$, then we can find a subset $\B_0\subseteq \A$ such that $|\B_0| \gg_{\delta} |\A|$, and all elements in $\B_0$ have the same largest square factor. In particular, after dividing all elements in $\B_0$ by their common largest square factor, the resulting set, denoted by $\B$, contains only square-free elements.
\end{lemma}
To prove this, we use the Pigeonhole Principle. We first show that most elements in $\P$ with relatively small $a$ and $d$ are not divisible by large perfect squares.
\begin{lemma}\label{lem:large-sqs}
        Let $\P = \{a+id: 0\leq i < L\}$ be an arithmetic progression with $\gcd(a, d) = 1$ and $a, d > 0$. Let $T > 0$ be a positive integer. The number of elements in $\P$ divisible by a square greater than $T^2$ is at most $\sqrt{a+dL} + \frac{L}{T}$.
\end{lemma}
\begin{proof}
    Let $n = \lfloor\sqrt{a+dL}\rfloor$. Clearly if an element in $\P$ is divisible by $t^2$, then $t\leq n$ (or otherwise $t^2$ is larger than any element in $\P$). We claim that for each $t$, the number of elements in $\P$ divisible by $t^2$ is at most $\lceil\frac{L}{t^2}\rceil$. Since $\gcd(a, d) = 1$,  if $t^2$ divides $a+i_1d$ and $a+i_2d$, then $\gcd(t, d) = 1$. This implies that $t^2|(i_1-i_2)$, so $t^2$ divides at most one element in every $t^2$ consecutive elements in $\P$. We conclude that the number of elements in $\P$ divisible by a square greater than $T^2$ is at most
    \[\sum_{t=T+1}^{n}\left\lceil\frac{L}{t^2}\right\rceil \leq \sum_{t=T+1}^{n}\left(1 + \frac{L}{t^2} \right)\leq n + L\sum_{t=T+1}^{n}\frac{1}{t^2} \leq n+\frac{L}{T}.\]
    We conclude with the desired bound by noting that $n \leq \sqrt{a+dL}$.
\end{proof}
\begin{proof}[Proof of Lemma~\ref{lem: squarefree}]
As $a+dL < \frac{\delta^2}{9}L^2$, by taking $T = \lceil\frac{3}{\delta}\rceil$ in Lemma~\ref{lem:large-sqs}, there are at least $\frac{\delta}{3}L$ elements left in $\A$ that are not divisible by squares greater than $T^2$. By the Pigeonhole Principle, there exists some $1\leq t \leq T$ such that there are at least $\frac{\delta L}{3T} \geq \frac{\delta^2}{18}L$ elements in $\A$ that are divisible by $t^2$ and are square-free after divided by $t^2$. We shall let $\B_0$ be the set of such elements.
\end{proof}

\begin{proof}[Proof of Theorem~\ref{thm: A energy 2log2-1} assuming Theorem~\ref{thm:reduced}]
First, we may assume that $L \geq |\A|$ is sufficiently large with respect to $\delta$, or otherwise we may choose $\A' = \A$ and we have $E_\times(\A) \ll |\A|^4 \ll_\delta 1$.

We finish proof by combining all results from above. Our setup is the following. Let $\P = \{a+id: 0\leq i < L\}$ be an arithmetic progression with common difference $d$ and length $L$. Let $\A\subseteq \P$  with density $\frac{|\A|}{|\P|} \geq \delta$ for some constant $\delta \in (0, 1)$. We begin by noticing the following simple observation.  
\subsection*{Observation} Let $\phi$ be a non-trivial dilation operation (multiply by a nonzero constant) on the set $\P$, then the operation $\phi$ preserves the multiplicative energy of any subset of $\P$.  
If Theorem~\ref{thm: A energy 2log2-1} holds for $(\A_0, \P_0)$ where $\phi(\A_0)\subseteq \A$ and arithmetic progression $\phi(\P_0) \subseteq \P$ with $|\A_0|\gg_{\delta} |\A|$, then the conclusion must also hold for $(\A, \P)$.

Based on the above observation, we do a set of reduction steps to reduce to $(\A', \P')$. At the same time, we trace the change of density from $\delta$ to $\delta'$ in each step and $L'$ to $L$ in each step.

\noindent
\textbf{Step 1: Assume all elements are positive.} 
Without loss of generality, we may assume that at least $\frac13$ of the elements in $\A$ are positive, or we shall replace $\A$ by $-\A$. Let $\A_+ = \A\cap \Z_{>0}$ be the set of positive elements.
Then it is sufficient to prove the statement for $\A_{+}$, i.e., one can find a subset of $\A_{+}$ with small multiplicative energy $\ll L^{2}(\log L)^{1-2\log2+o(1)}$. 

In this step, we set $\A_1 = \A_+$. It is clear that by doing this, the density $\delta_1 > \delta/3$. We also make the enclosing arithmetic progression contain only positive elements. This does not decrease the density $\delta_1$, while the length of the enclosing arithmetic progression becomes $L_1 : = |\P \cap \Z_{>0}| \geq |\A_1| = \Omega(\delta L)$.
In summary, we can find an arithmetic progression $\P_1 = \{a_1+id_1: 0 \leq i < L_1\}$ and a subset $\A_1\subseteq \P_1$ such that $E_\times(\A) \geq E_\times(\A_1)$, $a_1, d_1 > 0$, and $L_1 = \Omega(\delta L)$, $\delta_1 = \frac{|\A_1|}{|\P_1|} \geq \delta/3$.

\noindent
\textbf{Step 2: Divide all elements by GCD.}
We may divide all elements in $\A_1$ by $\gcd(a_1, d_1)$, giving set $A_2$.
The density and the length do not change $\delta_2=\delta_1>\delta/3$ and $L_2 = L_1 = \Omega(\delta L)$. 
The new common difference is $d_2 = \frac{d_1}{\gcd(a_1, d_1)}$ and the initial term becomes $a_2 = \frac{a_1}{\gcd(a_1, d_1)}$.
Moreover this map $(\times \gcd(a_1, d_1))$ from subsets of $\A_{2}$ to subsets of $\A_1$ preserves multiplicative energy.

\noindent
\textbf{Step 3: Assume all elements lie in dyadic interval \texorpdfstring{$[x,2x)$}{[x, 2x)}.}
Let $I_t:= \{a_2+id_2:L_2/2^{t+1}\leq i < L_2/2^t\}$ for $t\ge 0$
and consider the dyadic partition (ignore the single element $\{a\}$)
\[\{a_2+id_2:0 <  i < L_2\} = \bigcup_{t\ge  0}I_t .\] 
We may set $t_0 = O(1 + \log \frac{1}{\delta_2})$ such that $\sum_{t\ge  t_0} |I_t| \le \delta_2 (L_2-1) /2$. It implies that
\[\left|\A_2 \cap \bigcup_{t < t_0} I_t \right|\ge \delta_2 (L_2-1)/2.  \]
This means that $\A_2$ contains at least $\delta_2/2$ fraction of elements in $\bigcup_{t < t_0} I_t$. In particular there exists some $0\leq t_1 < t_0$ such that
\[|\A \cap I_{t_1} | \ge \frac{\delta_2}{2}|I_{t_1}| .  \]
Note that $I_{t_1} = \{a_2 + id_2: L_2/2^{t_1+1} \leq i < L_2/2^{t_1}\}$ can be written as an arithmetic progression $\P_3 = \{a_3+id_3:0\leq i < L_3\}$ with $a_3 \geq a_2 + d_2L_2/2^{t_1+1}$, $d_3 = d_2$, and $L_3 \leq L_2/2^{t_1+1}$. We conclude that we can find sets $\A_3,\P_3$ such that $\P_3 = \{a_3+ i d_3:0\le i < L_3\}$ with $a_3 > d_3L_3$, $L_3 = \Omega(L_2/2^{t_0}) = \delta_2^{O(1)}L_2$ and $\A_3= \A_2\cap \P_3$ with $|\A_3| \geq \frac{\delta_2}{2}L_3 = : \delta_3 L_3$. 
Thus, by the observation made at beginning, it suffices to study $(\A_3, \P_3)$.

\noindent
\textbf{Step 4: Eliminate extremal cases.} After step $3$, we may assume that $d_3L_3 < a_3$. Applying Lemma~\ref{lem:large-a},
we may assume that $d_3L_3 < a_3 < L_3 \log L_3$, otherwise we directly have $E_{\times}(\A_3) \ll_{\delta} |\A_3|^{2}$.

\noindent
\textbf{Step 5: Assume all elements are square-free.}
After reduction in step 2, the condition in Lemma~\ref{lem: squarefree} is satisfied and by applying Lemma~\ref{lem: squarefree}
we get a set $\A_5$ with size at least $\frac{\delta_3^{2}}{18}L_3$ containing only square-free elements. Moreover, there is some constant $T = O(1/\delta)$ such that $T^2\A_5\subseteq \A_3$ and $\gcd(T, d_3) = 1$. Clearly $\A_5$ is also contained in an arithmetic progression of length $L_5 = \Theta(L_3/T^2) = \Omega(\delta^2L_3)$. Also note that $L_5 \leq L_3$, so we have $\delta_5 \geq \frac{\delta_3^{2}}{18}$. 
 By applying the observation made at the beginning again, it is valid to make the deduction. 
 
\noindent
\textbf{Summary.}
Combing the five reduction steps above, we conclude that either we find a subset $\A'\subseteq \A$ with $|\A'|\gg_{\delta} |\A|$ and $E_{\times}(\A')\ll_{\delta} |\A|^{2}$ (which is more than what we need), or we get a set $\B=\A_5$ containing only square-free integers, with density $\delta_5 = \Omega(\delta^{2})$ in an arithmetic progression $\P' = \{a'+id':0\le i < L'\}$ satisfying that $L'\L L' > a' > d'L'>0$ where $L' = L_5 = \Omega( \delta^{O(1)} L)$. Applying Theorem~\ref{thm:reduced} on $\B$, there exists a subset $\B'\subseteq \B$ with size $|\B'|\gg |\B|$ such that
\[E_{\times}(\B') \ll |\B|^{2}(\log |\B|)^{2\log 2-1+o(1)}.\]
Moreover, it is guaranteed that there is a nonzero integer $m$ such that $m\B:= \{m\}\cdot \B \subseteq \A$.
Then $\A' = m\B'$ is a subset of $\A$ of size at least $\delta_5 L' \gg_{\delta} |\A|$, and
\[E_{\times }(\A') = E_{\times}(\B') \ll |\B|^{2}(\log |\B|)^{2\log 2-1+o(1)}\ll_{\delta}  |\A|^{2}(\log |\A|)^{2\log 2-1+o(1)},\]
as desired.\end{proof}


    

\section{Proof of Theorem~\ref{thm:reduced}:  the essential case}\label{sec:square-free}
In this section we prove Theorem~\ref{thm:reduced}. Let us restate the theorem for convenience.
\reduced*

In fact we prove the following more explicit and stronger statement.
\begin{theorem}\label{thm:reduced stronger}
For any $\delta, \epsilon \in (0, 1]$, there exists $L_0 = L_0(\delta, \epsilon)$ such that the following holds. Let $\P = \{a+id: i\in 0\leq i < L\}$ be an arithmetic progression with common difference $d$ and length $L$, with $L > L_0, d > 0$, $L\log L > a \geq dL$, and $\gcd(a, d) = 1$. If $\A\subseteq \P$ is a subset of size at least $\delta L$ containing only square-free elements, then there exists a subset $\A'\subseteq \A$ with $|\A'|\geq |\A|/2$ such that

\begin{equation}\label{eqn:thm-reduced-stronger}
 E_\times(\A') \ll L^2 \left( 1+\frac{L}{a}(\log L)^{2\log 2-1 + \epsilon} \right).
\end{equation}
\end{theorem}
 As $a \geq dL \geq L$, \eqref{eqn:thm-reduced-stronger} implies \eqref{eqn:thm-reduced}. Thus, Theorem~\ref{thm:reduced} follows from Theorem~\ref{thm:reduced stronger}.
 To prove Theorem~\ref{thm:reduced stronger}, we need the following result of Shiu \cite{Shiu80}.
\begin{theorem}[Theorem 1 in \cite{Shiu80}]\label{lem:Shiu}
		Let $f(n)$ be a non-negative multiplicative function such that $f(p^{\ell}) \le A_1^{\ell}$ for some positive constant $A_1$ and for any $\ee>0$, $f(n)\le A_2 n^{\ee}$ for some $A_2=A_2(\ee)$. Let $\alpha, \beta \in (0,1/2]$, integer $a$ satisfying $\gcd(a, k) = 1$. Then as $x\to \infty$ we have
		\[\sum_{ \substack{x-y \leq n < x\\ n \equiv a \Mod k}} f(n) \ll \frac{y}{\phi(k)}\frac{1}{\log x} \exp\left (\sum_{p\leq x, p\nmid k} \frac{f(p)}{p} \right), \]
		provided that $k<y^{1-\alpha}$ and $x^{\beta} < y < x$, where the implicit constant depends only on $A_1, A_2, \alpha, \beta$ and the summation on the right hand side is taken over prime $p$.
\end{theorem}
Let $\omega(n)$ be the number of distinct prime factors of $n$.
Consider the family of nonnegative multiplicative functions $M = \{z^{\omega(n)}: z\in (0, \upperz]\}$. Note that Theorem~\ref{lem:Shiu} is applicable to any $f\in M$ with a universal choice of $A_1$ and $A_2$, by observing that $f(p^\ell) \leq \upperz$, and $f(n) \leq \upperz^{\omega(n)} = n^{O(1/\log\log n)}$. To evaluate the summation inside the exp on the right hand side, we need the following estimate.
\begin{theorem}[Merten's estimate, e.g. see {\cite[Theorem 1.10]{Tenenbaum}}]\label{thm:merten}
For $x \geq 2$,
$\sum_{p\leq x}\frac{1}{p} = \LL x + O(1).$
\end{theorem} 

Using these, we estimate the number of elements in $\P$ with a large number of distinct prime factors.

\begin{lemma}
        Let $\P = \{a+id: 0\leq i < L\}$ be an arithmetic progression satisfying $L^\Lpower > a + dL > a > 0$ and $\gcd(a, d) = 1$. As $a+dL\to\infty$, if $t > 0$ satisfies $t = o(\log\log (a+dL))$, then
        \[|\{n\in \P: \omega(n) \geq \log\log(a+dL) + t\}| \ll \frac{dL}{\phi(d)}\exp\left(-\left(\frac{1}{2}+o(1)\right)\frac{t^2}{\log\log (a+dL)}\right).\]
\end{lemma}
\begin{proof}
    Let $\ell = \log\log (a+dL) + t$ and $z = \frac{\ell}{\log\log (a+dL)} = 1+o(1) \in (1, \upperz]$. As $z > 1$, we use bounds $z^{\omega(n)-\ell} \geq 1$ when $\omega(n) \geq \ell$ and $z^{\omega(n)-\ell} > 0$ otherwise, to get that
    \[|\{n\in \P: \omega(n) \geq \ell\}| \leq \sum_{n\in \P}z^{\omega(n)-\ell} = z^{-\ell}\sum_{n\in \P}z^{\omega(n)}.\]
    As $\ell = z\log\log (a+dL)$, we have $z^{-\ell} = e^{-z\log z\log\log (a+dL)} = (\log (a+dL))^{-z\log z}$. Note that we can write $\P = \{a \leq n < a+dL: n\equiv a\pmod d\}$. By $L^2 > a+dL > a > 0$, we may take $x = a+dL$, $y = dL$, and $k = d$, and they satisfy the conditions for Theorem~\ref{lem:Shiu} with $\alpha = \beta = 1/2$. It implies
    \[z^{-\ell}\sum_{n\in \P}z^{\omega(n)} \ll\frac{dL}{\phi(d)}(\log(a+dL))^{z-1-z\log z}.\]
    The implicit constant here is uniform for all $z\in (1, 2]$. If we write $\lambda := z-1 = o(1)$, then we have
    \[-1-z\log z+z = -1 - (1+\lambda)\log(1+\lambda) + (1+\lambda) = -\frac{\lambda^2}{2} + O(\lambda^3) = -\left(\frac{1}{2}+o(1)\right)\lambda^2.\]
    By definition, we have $\lambda = \frac{t}{\log\log (a+dL)}$, and the desired inequality follows.
\end{proof}
Note that $\frac{d}{\phi(d)} = O(\log\log d) = O(\log\log (a+dL))$. We can choose $t = (\log\log (a+dL))^{2/3}$, such that $o(L)$ elements have more than $T = \log\log (a+dL) + (\log\log (a+dL))^{2/3}$ distinct prime factors. Therefore we have the following lemma.
\begin{lemma}\label{lem:omega}
        For any $\ee \in (0, 1]$, there exists $a_0 = a_0(\ee)$ such that the following holds. For any $a > a_0$, if $\P = \{a+id: 0\leq i < L\}$ is an arithmetic progression with common difference $d$ and length $L$ satisfying $L^\Lpower > a + dL > a > 0$ and $\gcd(a, d) = 1$, then for $T = \log\log (a+dL) + (\log\log (a+dL))^{\frac{2}{3}}$,
\begin{equation}\label{eqn: deleting}
    |\{n\in \P: \omega(n) \leq T\}| \geq (1-\ee)L.
\end{equation}        
\end{lemma}



Let $\A'$ be the 
subset of elements of $\A$ with at most $T$ distinct prime factors, i.e.
\begin{equation}\label{def: A'}
    \A': = \{a\in \A:  \omega(n)\le T = \log\log (a+dL) + (\log\log (a+dL))^{\frac{2}{3}} \}.
\end{equation}
By Lemma~\ref{lem:omega}, we have $|\A'|\ge |\A|/2$ when $a > a_0(\delta/2)$. 
\begin{proof}[Proof of Theorem~\ref{thm:reduced stronger}]
    By choosing a sufficiently large constant in \eqref{eqn:thm-reduced}, we may assume that $a > dL\geq L > |\A|$ is sufficiently large. In particular $d < \frac{a}{L} < \log L < L^{1/3}$. By Lemma~\ref{lem:omega}, we may assume that $a\ge a_0(\delta/2)$. Now at most $\ee L = 
\frac{\delta}{2}L$ elements in $\P$ have more than $T = \log\log(a+dL) + (\log\log(a+dL))^{\frac{2}{3}}$ distinct prime factors. Similar to the proof of Lemma~\ref{lem:large-a}, we estimate the size of
    \[X := \{(x_1, x_2;y_1, y_2): x_iy_j\in \A'\;~~\forall \; i, j\in \{1, 2\}, 0 < x_1 < x_2, 0 < y_1 < y_2, x_1 \leq y_1\}.\]
    By the same argument as in the proof of Lemma~\ref{lem:large-a} (here we use an extra symmetry between $x$ and $y$), we have that $E_\times(\A') \leq 2|\A'|^2 + 8|X|$.
    
    Fix $x_1$. As all elements in $\A'$ are congruent to $a$ mod $d$, they are all coprime to $d$. Therefore $x_i$ and $y_j$ are also coprime to $d$. As a result, $a\equiv x_1y_1 \equiv x_2y_1 \equiv x_1y_2 \pmod d$ implies
    that $x_2$ is congruent to $x_1$ mod $d$ and $y_1$ and $y_2$ are both congruent to $ax_1^{-1}$ mod $d$. Since all elements in $\A'$ are included in the interval $[a, a+dL)$, one has $\frac{x_2}{x_1} = \frac{x_2y_1}{x_1y_1} < 1+\frac{dL}{a}$. Since $x_2 > x_1$, we know that $x_2\in (x_1, x_1+x_1dL/a)$. As $x_2\equiv x_1\pmod d$, the number of choices for $x_2$ is at most $\lfloor\frac{x_1L}{a}\rfloor$. In particular, if $\frac{x_1L}{a} < 1$ then there is no valid tuple in $X$.
    
    Similarly, as $x_1y_i\in [a, a+dL)$, we know that $y_i\in [\frac{a}{x_1}, \frac{a}{x_1}+\frac{dL}{x_1})$. As $y_1 < y_2$ and they are congruent mod $d$, we know that there is no valid choice for $(y_1, y_2)$ if $\frac{L}{x_1} < 1$. Otherwise if $\frac{L}{x_1} \geq 1$, then the number of choices is at most 
    \[\frac{1}{2}\left\lceil\frac{L}{x_1}\right\rceil\cdot \left(\left\lceil\frac{L}{x_1}\right\rceil - 1\right) \leq \frac{L^2}{x_1^2}.\]
    Therefore, for each $x_1$, the number of choices for $(x_2; y_1, y_2)$ is at most $\frac{L^3}{ax_1}$. 
    
    Let $X_1$ be the set of such tuples with $x_1 \leq \frac{ea^2}{L^2}$. We first estimate $|X_1|$. As we sum $\frac{L^3}{ax_1}$ over $1\leq x_1 \leq \frac{ea^2}{L^2}$, we conclude that $|X_1| \leq \frac{L^3}{a}\left(\log \frac{ea^2}{L^2}+1\right) \leq L^2\left(2\frac{L}{a} + 2\frac{L}{a}\log\frac{a}{L}\right) \leq 4L^2.$
    
    It remains to estimate the size of $X_2 = X\setminus X_1$. It consists of tuples with $x_1 > \frac{ea^2}{L^2}$. Note that $x_1^2 \leq x_1y_1 < a+dL$, we know that $\frac{ea^2}{L^2} < x_1 < \sqrt{a+dL}$.
    For each fixed $x_1 = x\in (\frac{ea^2}{L^2}, \sqrt{a+dL})$, we consider tuples $(x_1, x_2; y_1, y_2)\in X_2$. As argued above, we know that $x_2$ is an element in $(x_1, x_1dL/a)$ with $x_2\equiv x_1\pmod d$, and $y_1, y_2$ are elements in $[\frac{a}{x_1}, \frac{a}{x_1}+\frac{dL}{x_1})$ with $y_1\equiv y_2 \equiv ax_1^{-1}\pmod d$.
    
    Next we use the fact that all elements in $\A'$ having at most $T$ distinct prime factors and are square-free, which implies that $\omega(x_i)+\omega(y_j) = \omega(x_iy_j) \leq T$ for $i, j\in \{1, 2\}$. Therefore for each $(x_1, x_2; y_1, y_2)\in X_2$, we have
    \[ \left(\frac{1}{\sqrt{2}}\right)^{\omega(x_1y_1)+\omega(x_1y_2) + \omega(x_2y_1) + \omega(x_2y_2) - 4T}  \geq 1.\]
    Using the above bound and that $\omega(x_iy_j) = \omega(x_i)+\omega(y_j)$ since $x_iy_j\in \mathcal{A}'$ is square-free, we have
    \begin{equation}\label{eqn:thm-reduced-as-product}
    \begin{split}
        |X_2| & \leq \sum_{(x_1, x_2; y_1, y_2)\in X_2}
    2^{-\omega(x_1) -\omega(x_2)- \omega(y_1)-\omega(y_2)+2T } \\
        & \leq 2^{2T}\sum_{\frac{ea^2}{L^2} < x_1 < \sqrt{a+dL}}2^{-\omega(x_1)}\sum_{\substack{x_1 < x_2 < x_1\frac{a+dL}{a}\\ x_2\equiv x_1\Mod d}}2^{-\omega(x_2)}\left(\sum_{\substack{\frac{a}{x_1}\leq y < \frac{a+dL}{x_1}\\ y\equiv ax_1^{-1}\Mod d}}2^{-\omega(y)}\right)^2.
    \end{split}
    \end{equation}
    We now apply Theorem~\ref{lem:Shiu} to two factors on the right hand side of \eqref{eqn:thm-reduced-as-product}. Note that for $x' = x_1\frac{a+dL}{a}$ and $y' = x_1dL/a$, we have $y' > \frac{ea^2}{L^2}\frac{dL}{a} = \frac{ea}{L}d \geq ed^2 > d^2$ (where we use the assumption that $a \geq dL$) and $x' > y'$. Finally we have $(y')^2 = x_1^2d^2L^2/a^2 > ex_1d^2 \geq ex_1 > x'$. Therefore parameters $(x, y, d, f) = (x', y', k, 2^{-\omega(\cdot)})$ satisfy the conditions of Theorem~\ref{lem:Shiu} with $\alpha = \beta = 1/2$ and $A_1 = A_2 = 1$. By Theorem~\ref{thm:merten}, we have $\sum_{p\leq x', p\nmid d}f(p)/p \leq \frac{\LL x'+C}{2}$ for $f(x) = 2^{-\omega(x)}$, and further
    \begin{equation}\label{eqn: omega }
    \sum_{\substack{x_1 < x_2 < x_1\frac{a+dL}{a}\\ x_2\equiv x_1\Mod d}}2^{-\omega(x_2)} \ll \frac{x_1\frac{dL}{a}}{\phi(d)}\left(\log\frac{x_1(a+dL)}{a}\right)^{-\frac{1}{2}} \leq \frac{x_1dL}{a\phi(d)}(\log x_1)^{-\frac{1}{2}}.
    \end{equation}
    Similarly as $x_1 < \sqrt{a+dL} \leq \sqrt{2L\log L} < L^{\frac23}$ for $L$ sufficiently large, one can verify that for $(x', y') = (\frac{a+dL}{x_1}, \frac{dL}{x_1})$, we have $(y')^2 = d^2\frac{L^2}{x_1^2} > d^2L^{\frac23} > d^4$, $x' > y'$, and $\frac{(y')^2}{x'} = \frac{d^2L^2}{x_1(a+dL)} > \frac{d^{2}L^2}{(a+dL)^{3/2}} > 1$. Hence by Theorem~\ref{lem:Shiu} with the same $(\alpha,\beta,A_1,A_2)$ and by the same estimate due to Theorem~\ref{thm:merten}, we have
    \begin{equation}\label{eqn: omega y}
        \sum_{\substack{\frac{a}{x_1}\leq y < \frac{a+dL}{x_1}\\ y\equiv ax_1^{-1}\Mod d}}2^{-\omega(y)} \ll \frac{\frac{dL}{x_1}}{\phi(d)}\left(\log\frac{a+dL}{x_1}\right)^{-\frac12} \leq \sqrt{2}\frac{dL}{x_1\phi(d)}(\log(a+dL))^{-\frac{1}{2}}.
    \end{equation}
    Here in the second inequality we use $\frac{a+dL}{x_1} \geq \sqrt{a+dL}$. Put \eqref{eqn: omega } and \eqref{eqn: omega y} back to \eqref{eqn:thm-reduced-as-product}. We get
    \begin{equation}\label{eqn:thm-reduced-sum}
    \begin{split}
        |X_2|&\ll 2^{2T}\sum_{\frac{ea^2}{L^2} < x_1 < \sqrt{a+dL}}2^{-\omega(x_1)}\cdot \frac{d^3L^3}{x_1a\phi(d)^3}(\log x_1)^{-\frac12}(\log(a+dL))^{-1}\\
        & = \frac{d^3L^3}{a\phi(d)^3}2^{2T}(\log(a+dL))^{-1}\sum_{\frac{ea^2}{L^2} < x_1 < \sqrt{a+dL}}2^{-\omega(x_1)}\frac{(\log x_1)^{-\frac12}}{x_1}.
    \end{split}
    \end{equation}
    Here the implicit constant is absolute. To estimate the second line of \eqref{eqn:thm-reduced-sum}, we partition the interval dyadically over $x_1\in (e^k, e^{k+1})$ for $1\le \lfloor\log\frac{ea^2}{L^2}\rfloor \leq k \leq \lfloor\log\sqrt{a+dL}\rfloor$.  In each interval we have
    \[\sum_{e^k < x_1 < e^{k+1}}2^{-\omega(x_1)}\frac{(\log x_1)^{-\frac12}}{x_1} \leq \sum_{e^k < x_1 < e^{k+1}}2^{-\omega(x_1)}\frac{k^{-\frac12}}{e^k} \ll e^k(k+1)^{-\frac12}\frac{k^{-\frac12}}{e^k} \ll k^{-1},\]
    where in the second inequality we use Theorem~\ref{lem:Shiu} with $(x, y, d, a) = (e^{k+1}, e^{k+1}-e^k, 1, 1)$ and parameters $\alpha=\beta=1/2$ and $A_1=A_2 = 1$. Summing over all values of $k$, we have for large enough $L$,
    \begin{equation}\label{eqn: x_1}
      \sum_{\frac{ea^2}{L^2} < x_1 < \sqrt{a+dL}}2^{-\omega(x_1)}\frac{(\log x_1)^{-\frac12}}{x_1} \ll \sum_{k=\lfloor\log\frac{ea^2}{L^2}\rfloor}^{\lfloor\log\sqrt{a+dL}\rfloor}k^{-1} \leq 2 \log\log(a+dL) . 
    \end{equation}
Noting that $T = \log\log(a+dL) + (\log\log(a+dL))^{\frac23}$ and by choosing $L_0$ sufficiently large, we have 
\begin{equation}\label{eqn: 2T}
    2^{2T} = (\log (a+dL))^{(2\log 2)\left(1+(\log\log(a+dL))^{-\frac13}\right)} < (\log (a+dL))^{2\log 2 + \epsilon/3}.
\end{equation}
    Put \eqref{eqn: x_1} and \eqref{eqn: 2T} back to \eqref{eqn:thm-reduced-sum}. Noting that $\log\log(a+dL) < (\log(a+dL))^{\epsilon/3}$, $\frac{d}{\phi(d)} = O(\log\log d) \leq (\log(a+dL))^{\epsilon/9}$, and $a+dL \leq 2L\log L \leq L^2$ as $L\geq L_0$ is sufficiently large, we conclude that
    \[|X_2| \ll \frac{L^3}{a}\left(\frac{d}{\phi(d)}\right)^3(\log(a+dL))^{2\log2-1 + \epsilon/3}(\L(a+dL))^{\epsilon/3} \ll \frac{L^3}{a}(\log L)^{2\log2-1 + \epsilon}.\]
Here the implicit constant factor is absolute. It follows that
    \[E_\times(\A') \leq 2|\A'|^2 + 8(|X_1|+|X_2|) \ll L^2\left(1+\frac{L}{a}(\log L)^{2\log2-1 + \epsilon}\right).\]
    Recall that $|\A'| \geq |\A|/2$. We have the desired subset $\A'$.
\end{proof}

\section{Elements with a fixed number of primes factors in arithmetic progressions}\label{sec: Smirnov}

In this section, we make some preparations for the proof of Theorem~\ref{thm: A energy 86} in Section~\ref{Sec: main thm}. Let $\A$ be an arithmetic progression satisfying $0 < dL \leq a \leq L\sqrt{\log L}$. For any positive integer $n$, we denote the sequence of distinct prime factors of $n$ by $p_1(n) < p_2(n) < \cdots < p_{\omega(n)}(n)$. For parameters $\alpha, \beta\in \R$ and $k\in \N$, we define
\begin{equation}\label{eqn: N}
    \CN_k(\A; \alpha, \beta) := \{n \in \A: n\textrm{  square-free},\; \omega(n) = k,\; \LL p_j(n) \geq \alpha j-\beta\;~~ \forall\;1\leq j\leq k\}
\end{equation}
and let $N_k(\A; \alpha, \beta) = |\CN_k(\A; \alpha, \beta)|$. Our goal is to give a lower bound on $N_k(\A; \alpha, \beta)$. This is a natural generalization of \cite{Erd60,Fordstats07, Ten84} where the object studied is the special arithmetic progression $\A = [N]$. Our proof uses ideas from \cite{Erd60, Ten84} and in Appendix~\ref{sec:Smirnov-stronger} we give a stronger estimate following \cite{FordExtremal18}.

First we estimate the number of primes in an arithmetic progression using the following theorem of Siegel and Walfisz. For a finite $\A\subseteq \N$, let $\pi(\A)$ be the number of primes in $\A$.
\begin{theorem}[Siegel-Walfisz, see Section II.8. in \cite{Tenenbaum}]\label{thm:SW}
Let $\pi(x; q, a)$ be the number of primes up to $x$ with $p\equiv a \pmod q$. Then for any $A>0$, 
uniformly for $(a, q) = 1$ and $1\leq q \leq (\L x)^A$, we have
\begin{equation}\label{eqn: pi}
    \pi(x; q, a) = \frac{x}{\phi(q) \log x} + O\left(\frac{x}{\phi(q)(\L x)^2}\right).
\end{equation}
\end{theorem}
\begin{lemma}\label{lem:primes-in-AP}
If $\A = \{a+id:0\leq k < L\}$ satisfies that $L$ is sufficiently large, and $0 < dL < a < 10L\sqrt{\log L}$ and $\gcd(a, d) = 1$, then $\pi(\A) \geq \frac{dL}{2\phi(d)\log L}$.
\end{lemma}
\begin{proof}
First note that we can write $\pi(\A) = \pi(a+dL-1; d, a) - \pi(a-1; d, a)$. From our assumption we know that $d < 10\sqrt{\log L}$ and $a+dL > a > L$, so we may take $A = 1$ in Theorem~\ref{thm:SW} and $L > e^{100}$ such that $1\leq d < \log (a-1) < \log (a+dL-1)$. Hence we apply \eqref{eqn: pi} to have
\[  \pi(\A) =  \frac{a+dL-1}{\phi(d)\log (a+dL-1)} - \frac{a-1}{\phi(d)\log (a-1)} +   O\left(\frac{a-1}{\phi(d)(\log (a-1))^2}+ \frac{a+dL-1}{\phi(d)(\log (a+dL-1))^2} \right). \] 
To estimate the main term, noticing that  $a-1\leq a+dL-1 \leq 2a-2$, we have
\[\frac{a+dL-1}{\phi(d)\L (a+dL-1)} = \frac{a+dL-1}{\phi(d)\L (a-1)} + O\left(\frac{a+dL-1}{\phi(d)(\L (a-1))^2}\right).\]
This shows that
\[\pi(\A) = \frac{dL}{\phi(d)\L (a-1)} + O\left(\frac{a-1}{\phi(q)(\log (a-1))^2} + \frac{a+dL-1}{\phi(d)(\log (a+dL-1))^2} + \frac{a+dL-1}{\phi(d)(\L (a-1))^2}\right).\]
Clearly $\frac{a+dL-1}{\phi(d)(\L (a-1))^2}$ is the largest among the three summands in the error term. As $L \leq dL < a \leq 10L\sqrt{\log L}$, we can bound the error terms by
\[\frac{a+dL-1}{\phi(d)(\L (a-1))^2} \leq \frac{20dL\sqrt{\L L}}{\phi(d)(\log L)^2} = O\left(\frac{dL}{\phi(d)(\L L)^{\frac 32}}\right).\]
As $L$ is sufficiently large, we may assume that the error term is at most $\frac{1}{6}\frac{dL}{\phi(d)\L L}$. Also note that $a-1\leq 10L\sqrt{\log L} \leq L^{\frac32}$ when $L $ is sufficiently large. This gives the desired estimate
\[\pi(\A)\geq \frac{2}{3}\frac{dL}{\phi(d)\log L} - \frac{1}{6}\frac{dL}{\phi(d)\L L} = \frac{dL}{2\phi(d)\L L }.\]
\end{proof}
Using this, we can estimate $N_k(\A;\alpha, \beta)$ with a specific choice of the parameters motivated by \cite{Erd60}.

\begin{proposition}\label{Prop: N}
 If $\A = \{a+id:0\leq i < L\}$ is an arithmetic progression satisfying that $L$ is sufficiently large, $0 < dL \leq a \leq L\sqrt{\log L}$, and $\gcd(a, d) = 1$, then for $k = \left\lfloor\frac{\LL L}{\L 4} -5\sqrt{\L\L L}\right\rfloor -4 $,
\[N_k(\A; \log 4, 1) 
\gg {L}{(\log L)^{-\theta -o(1)} }.
\] 
\end{proposition}

\begin{proof}
 To obtain a lower bound, we count elements $n\in \CN_k(\A; \log 4, 1)$ with $p_1(n)p_2(n)\cdots p_{k-1}(n) < \sqrt{a}$.
 Once we fixed such a choice of $(p_1, p_2, \dots, p_{k-1})$ with $\LL p_j \geq j\log 4 - 1$ for all $1\le j \le k-1$, we count the number of $p_k$ with $p_k > p_{k-1}$ and $p_1p_2\cdots p_k\in \A$. Let $q = p_1p_2\cdots p_{k-1}$. Clearly if $qp\in \A$ for some prime $p$, then $p \geq \frac{a}{q} > \sqrt{a} > q > p_{k-1}$. Moreover we have \[\LL p > \LL \sqrt{L} = \LL L - \L 2 >  k \log 4 - 1,\] 
so any such choice of prime $p$ would give a valid choice for $p_k$ as defined in \eqref{eqn: N}. 

For any such fixed choice of $q$, if $\gcd(q, d) > 1$, then any element in $\A$ is not a multiple of $q$, so there is no such choice of $p_k$. 
If $\gcd(q, d) = 1$, then $q|a+id$ is equivalent to $i\equiv -ad^{-1}\pmod q$ and there is an unique such $i = i_0$ with $0\leq i_0 < q$. Then the multiples of $q$ in $\A$ are of the form $$\left\{a+(i_0+k'q)d:0\leq k' < \frac{L-i_0}{q}\right\} \supseteq \{a+(i_0+k'q)d:0\leq k' < \lfloor L/q\rfloor\},$$ and the inclusion follows from $\frac{L-i_0}{q} > \lfloor L/q\rfloor - 1$.
Hence any element $p\in \A_q := \{a'+k'd:0\leq k' < L'\}$ with $a' = \frac{a+i_0d}{q}$ and $L' = \lfloor L/q\rfloor$ would satisfy $pq\in \A$. It is sufficient to estimate the number of primes from $\A_q$, which itself is an arithmetic progression. As $\gcd(a', d) \leq \gcd(a+i_0d, d) = \gcd(a, d) = 1$, we have $\gcd(a', d) = 1$. Moreover, we know that $dL' \leq \frac{dL}{q} < \frac{a}{q} \leq a'$ and $a' < 10L'\sqrt{\log L'}$ when $L$ is sufficiently large. Note that $L' \geq \frac{L}{q} - 1 > \frac{L^{\frac12}}{(\L L)^{\frac14}} - 1$ is also sufficiently large. By invoking Lemma~\ref{lem:primes-in-AP}, the number of choices for $p_k$ is at least
\begin{equation}\label{eqn:count-last-prime}
    \pi(\A_q) \geq \frac{dL'}{2\phi(d)\L L'} \geq \frac{dL/q}{4\phi(d)\L L}.
\end{equation}
Summing \eqref{eqn:count-last-prime} up over all possible $(p_1, \dots, p_{k-1})$,
we have 
\begin{equation}\label{eqn:prop-N-target}
    N_k(\A; \log 4, 1) \geq \frac{dL}{4\phi(d)\log L}\cdot \sum_{\substack{p_1 < \cdots < p_{k-1}, p_1\cdots p_{k-1} < \sqrt{a}\\ p_j\nmid d,\LL p_j \geq j\log 4 - 1\;\forall\;1\leq j\leq k-1}}\frac{1}{p_1\cdots p_{k-1}}.
\end{equation}
Since $L \le a \le L^{2}$, the condition in Lemma~\ref{lem:sum-of-reciprocal} is satisfied with our choice of $k$, 
and it implies that
\[\begin{split}
N_k(\A; \log 4, 1) & \geq \frac{dL}{4\phi(d)\log L}(e\log 4)^{k-1}(\log \sqrt{a})^{-o(1)} \geq \frac{L}{(\log L)^{\theta+o(1)}} ,  
\end{split}
\]
where we use that $\frac{d}{\phi(d)}\geq 1$, $k = (1+o(1))\frac{\LL L}{\log 4}$, and the constant $\frac{1}{4}$ is absorbed in the $o(1)$ term.
\end{proof}

\begin{lemma}\label{lem:sum-of-reciprocal}
For any $x$ large enough, $k \le  \frac{\LL x}{\log 4} -5\sqrt{\log \log x}$, $d = O(\L x)$, we have
\begin{equation}\label{eqn:sum-of-reciprocal}
    \sum_{\substack{p_1 < \cdots < p_{k}, p_1\cdots p_{k} < x\\ p_j\nmid d,\LL p_j \geq j\log 4\;\forall\;1\leq j\leq k}}\frac{1}{p_1\cdots p_{k}} 
\ge (e\log 4)^k(\log x)^{- o(1)}.
\end{equation}
\end{lemma}

\begin{proof}
Let $h = \left\lfloor \frac{\sqrt{\log \log x}}{\log 4}\right\rfloor$, and $J = \floor{k/h} = \sqrt{\log \log x} -  O(1)$. For each $j\geq 1$, we define
\[I_j : = \left[ e^{e^{j \sqrt{\log \log x}
}}, e^{e^{(j+1) \sqrt{\log \log x}
}} \right). \]
By Theorem~\ref{thm:merten}, we know that $\sum_{p\in I_j}1/p = \sqrt{\LL x} + O(1)$. Therefore $\sum_{p\in P_j}1/p = \sqrt{\LL x} - O(\LLL x)$ where $P_j$ is the set of primes in $I_j$ which are not factors of $d$, and the 
$O(\LLL x)$ term bounds the contribution from those primes which are factors of $d$.
In the summation in \eqref{eqn:sum-of-reciprocal}, we consider only the tuples $(p_1, \dots, p_k)$ with $h$ primes in $P_j$ for each $1\leq j\leq J$ and $(k-Jh)$ primes in $P_{J+1}$. Noting that $k-Jh\leq h \leq e^h \leq e^{\sqrt{\LL x}}$, a straightforward computation shows that
\[p_1\cdots p_k \leq \exp\left(\sum_{j=1}^Jhe^{(j+1)\sqrt{\LL x}} + (k-Jh)e^{(J+2)\sqrt{\LL x}}\right) < e^{e^{(J+4)\sqrt{\LL x}}} < x\]
where in the last step we use that $J+4 < \sqrt{\LL x}$ for $x$ large. Hence, any such tuple would satisfy $p_1\cdots p_k < x$. Moreover, we know that $\LL p_j \geq \ceil{j/h}\sqrt{\LL x} > j\log 4$. Therefore any such tuple contributes to a summand on the left hand side of \eqref{eqn:sum-of-reciprocal}. Then the left hand side of \eqref{eqn:sum-of-reciprocal} is at least
\begin{equation}\label{eqn:sum-of-re-1}
    \prod_{j=1}^J\left(\sum_{p_1 < \cdots < p_h \in P_j}\frac{1}{p_1\cdots p_h}\right)\cdot \left(\sum_{p_1 < \cdots < p_{k-Jh}\in P_{J+1}}\frac{1}{p_1\cdots p_{k-Jh}}\right).
\end{equation}
Factors in \eqref{eqn:sum-of-re-1} are of the same form. We estimate each of them as follows. For $1\leq t\leq h$ and $j\geq 1$,
\begin{equation}\label{eqn:sum-of-re-each}
    \sum_{p_1 < \cdots < p_t\in P_j}\frac{1}{p_1\cdots p_t} = \frac{1}{t!}\sum_{p_1\in P_j}\frac{1}{p_1}\sum_{\substack{p_2\in P_j\\p_2\ne p_1}}\frac{1}{p_2}\cdots \sum_{\substack{p_t\in P_j\\p_t\ne p_1, \dots, p_{t-1}}}\frac{1}{p_t} \geq \frac{1}{t!}\left(\sum_{p\in P_j}\frac{1}{p}-\frac{t}{e^{e^{j\sqrt{\LL x}}}}\right)^t.
\end{equation}

Note that $t \leq h \le  \sqrt{\LL x}/\log 4 \leq e^{e^{j\sqrt{\LL x}}}$, so the term inside the parenthesis on the right hand side of \eqref{eqn:sum-of-re-each} is $\sqrt{\LL x} - O(\LLL x)$. We shall apply Stirling's formula and obtain that, when $x$ is sufficiently large, the right hand side of \eqref{eqn:sum-of-re-each} is at least
\[\begin{split}
  \textrm{RHS of }\eqref{eqn:sum-of-re-each} & \geq  \left(\frac{\sqrt{\log \log x} -O(\log \log \log x) - 1 }{t/e}\right)^{t} \cdot (\log \log x)^{-1/4-o(1)} \\
  & \geq  \left (e \log 4-  O\left(  \L\L\L x /\sqrt{\L\L x} \right)\right)^{t} \cdot (\log \log x)^{-1/4-o(1)} \\
  & \geq (e \log 4)^{t} \cdot (\log \log x) ^{-C}
\end{split}  \]
for some absolute constant $C>0$. 
Applying this bound to \eqref{eqn:sum-of-re-each} with $t= h$ and $t = k-Jh$, we have
\[ \begin{split}
\textrm{LHS of }\eqref{eqn:sum-of-reciprocal} \geq \eqref{eqn:sum-of-re-1} & \geq (e \log 4)^{Jh} (\log \log x)^{-CJ} \cdot  (e \log 4)^{k-Jt} (\log \log x)^{-C} \\
& = (e \log 4)^{k} \cdot (\L\L x)^{-O(\sqrt{\log \log x})}
\end{split}
\]
 Thus, noting that the $(\L\L x)^{-O(\sqrt{\log \log x})} = (\log x)^{-o(1)}$, we finish the proof. 
\end{proof}


\section{Proof of Theorem~\ref{thm: A energy 86}}\label{Sec: main thm}
Following the argument in Section \ref{sec:reduction}, we may only focus on the case where our arithmetic progression $\A =\{a+id:0\leq i < L\}$ satisfies $0 < dL < a < L\L L$. Moreover by Theorem~\ref{thm:reduced stronger} with $\epsilon = \frac{3}{2}-2\log2$, we may further assume that $a < \frac{1}{2}L\sqrt{\L L}$. In particular, we have the following reduction.

\begin{proposition}\label{prop:reduced-for-AP}
If $\A = \{a+id: 0\leq i < L\}$ is an arithmetic progression with length $L$ and common difference $d$ with $a > 0$, $L$ sufficiently large, and $L\sqrt{\log L} \leq a+dL$, then there exists a subset $\A'\subseteq \A$ satisfying $|\A'|\gg |\A|$ and $E_{\times}(\A') \ll |\A|^{2} $.
\end{proposition}

Therefore we may reduce to the case where $\A = \{a+id: 0\leq i < L\}$ with $2dL\le dL + a \leq L\sqrt{\log L}$. Inspired by \cite{FordExtremal18}, we set $A' = \CN_k(\A; \log 4, \beta)$ where parameters are as in Proposition~\ref{Prop: N}. This gives the desired lower bound on $|\A'|$. 
We show that $E_{\times}(\CN_k( \A; \log 4, \beta))$ is small. 

\begin{proposition}\label{thm: energy bound}
 Let $\A = \{a+id:0\leq i < L\}$ be an arithmetic progression with $L\sqrt{\log L} > a > dL > 0$ and $\A' = \CN_k(\A; \log 4, \beta)$ be defined as in \eqref{eqn: N} with $\beta \in \R$. Then for some absolute constant $C$ we have 
\[E_{\times} (\A')  \le CL^{2}(\log L)^{-2+\frac{1}{\log 2}} \cdot  \log \log L\cdot (\log 4)^{2k}2^{2\beta}.  \]
\end{proposition}

\begin{lemma}\label{lem:random-selection}
For any finite $\A\subseteq \R$, there exists $\A'\subseteq \A$ with $|\A'|\geq \frac{|\A|^3}{2E_\times(\A)}$ and $E_\times(\A') \leq 4|\A'|^2$.
\end{lemma}
\begin{proof}[Proof of Theorem~\ref{thm: A energy 86} assuming Proposition~\ref{thm: energy bound} and Lemma~\ref{lem:random-selection}]
We may assume that $L$ is sufficiently large and choose $\A'' = \CN_k(\A; \log 4, 1)$ where $k = \left\lfloor\frac{\LL L}{\L 4} -5\sqrt{\L\L L}\right\rfloor -4 = (1+o(1))\frac{\LL L}{\L 4}$. By Proposition~\ref{Prop: N}, we know that $|\A''| \geq {L}{(\log L)^{-\theta-o(1)}}.$ By Proposition~\ref{thm: energy bound}, noting that $(\log 4)^{2k} = (2\log 2)^{(1+o(1))\frac{\LL L}{\log 2}} = (\log L)^{1+\frac{\LL 2}{\log 2}+o(1)}$, we have
\[E_\times(\A'') \leq L^2(\log L)^{-2+\frac{1}{\log 2} + 1+\frac{\LL 2}{\log 2}+o(1)} = L^2(\log L)^{-2\theta+o(1)}.\]
We apply Lemma~\ref{lem:random-selection} to find $\A'\subseteq \A''$ of size at least $\frac{|\A''|^3}{2E_\times(\A'')} \geq L(\log L)^{-\theta-o(1)}$, as desired.
\end{proof}
The proof of Lemma~\ref{lem:random-selection} uses the probabilistic method. A similar application appears in \cite{FordExtremal18}.
\begin{proof}[Proof of Lemma~\ref{lem:random-selection}]
Fix $p = \frac{|\A|^2}{E_\times(\A)}\in (0, 1)$. We choose a random subset $\A'\subseteq \A$ by keeping each element in $\A$ independently with probability $p$. Then we have $\E |\A'|^2 \geq p^2|\A|^2$. Note that for tuples $(a_1, a_2;a_3, a_4)\in (\A')^4$ with $a_1a_2=a_3a_4$, there are at most $2|\A'|^2$ of them with $\{a_1, a_2, a_3, a_4\}$ containing at most $2$ elements. For all other ones, they come from a tuple $(a_1, a_2; a_3, a_4)\in \A^4$ with at least $3$ distinct elements, and they are included in $(\A')^4$ with probability at most $p^3$. Therefore we have
\[\E E_\times(\A')\leq \E 2|\A'|^2 + p^3E_\times(\A) = 2p^2|\A|^2 + p^3E_\times(\A) = 3p^2|\A|^2.\]
This shows that there exists $\A'$ such that $4|\A'|^2 - E_\times(\A')\geq \E[4|\A'|^2 - E_\times(\A')] = p^2|\A|^2$. In particular we know that $|\A'|^2 \geq \frac{p^2|\A|^2}{4} = \frac{|\A|^6}{4(E_\times(\A))^2}$ and $E_\times(\A') \le 4|\A'|^2$. 
\end{proof}

 The proof of Proposition~\ref{thm: energy bound} is similar to that of Theorem~\ref{thm:reduced stronger} but with two undetermined parameters that need to be optimized. 

\begin{proof}[Proof of Proposition~\ref{thm: energy bound}]
By letting the constant $C$ large, we may assume that $L$ is sufficiently large. Similar to the proof of Lemma~\ref{lem:large-a}, we upper bound the size of
    \[X := \{(x_1, x_2;y_1, y_2): x_iy_j\in \A'\;~~\forall \; i, j\in \{1, 2\}, 0 < x_1 < x_2, 0 < y_1 < y_2, x_1 \leq y_1\}.\]
    By the same argument as in Lemma~\ref{lem:large-a} (here we use an extra symmetry between $x$ and $y$), we have that $E_\times(\A') \leq 2|\A'|^2 + 8|X|$. It remains to estimate $|X|$.
    
    Fix $x_1$. As all elements in $\A'$ are congruent to $a$ mod $d$, they are all coprime to $d$. Therefore $x_i$ and $y_j$ are also coprime to $d$. As a result, $x_1y_1\equiv a \equiv x_1y_2 \pmod d$ implies that $x_1$ and $x_2$ are congruent mod $d$. Moreover, $y_1$ and $y_2$ are both congruent to $ax_1^{-1}$ mod $d$. Also note that elements in $\A'$ are included in the interval $[a, a+dL)$, we know that $\frac{x_2}{x_1} = \frac{x_2y_1}{x_1y_1} < 1+\frac{dL}{a}$. Hence also note that $x_2 > x_1$, we know that $x_2\in (x_1, x_1+x_1dL/a)$. Since $x_2\equiv x_1\pmod d$, there is no valid tuple in $X$ for $\frac{x_1L}{a}\le 1$. Thus any valid tuple satisfies $x_1 > a/L$. Similarly, as $x_1y_i\in [a, a+dL)$, we know that $y_i\in [\frac{a}{x_1}, \frac{a}{x_1}+\frac{dL}{x_1})$, and $y_1, y_2$ are congruent mod $d$. By the extra symmetry, we have $x_1 \leq \sqrt{x_1y_1} < \sqrt{a+dL}$.
    
    
    
    The goal is to apply Theorem~\ref{lem:Shiu} to bound the number of choices for $x_i$ and $y_j$. For each fixed $x_1 = x\in (L/a, \sqrt{a+dL})$, we would like to compute the number of tuples $(x_1, x_2; y_1, y_2)\in X$. Let us denote the set of such tuples by $X(x_1)$.
    
    By definition, any $m\in \A'$ satisfies that $\omega(m) = k$ and $\omega(m, t)\leq f(t) := \frac{\LL t}{\log 4} + \beta$ for all $t\in \N$. Here $\omega(m, t) : = \#\{p | m : p\le t \}.$    Note that we have $x_iy_j\in \A'$. As $\A'$ consists of square-free elements having $k$ distinct prime factors, we know that for any $(i, j)$, $k = \omega(x_iy_j) = \omega(x_i)+\omega(y_j)$, and $\omega(x_i, t)+\omega(y_j, t) = \omega(x_iy_j, t)\leq f(t)$. Let $\lambda, \xi\in (0, 1)$ to be determined. Then we have for each $(x_1, x_2; y_1, y_2)\in X_2$,
    \[(\lambda^2)^{\omega(x_1)}(\lambda^2)^{\omega(x_2)}(\lambda^2)^{\omega(y_1)}(\lambda^2)^{\omega(y_2)}\lambda^{-4k} = \lambda^{\omega(x_1y_1)+\omega(x_1y_2) + \omega(x_2y_1) + \omega(x_2y_2) - 4k}  = 1,\]
    and for any $t\in \N$,
    \[(\xi^2)^{\omega(x_1, t)}(\xi^2)^{\omega(x_2, t)}(\xi^2)^{\omega(y_1, t)}(\xi^2)^{\omega(y_2, t)}\xi^{-4f(t)} = \xi^{\omega(x_1y_1, t)+\omega(x_1y_2, t) + \omega(x_2y_1, t) + \omega(x_2y_2,t) - 4f(t)}  \geq 1.\]
    For each $x_1$, as $x_2 < \frac{a+dL}{a}x_1 < 2x_1$, if we choose $t = 2x_1$, then $\omega(x_i, t) = \omega(x_i)$ for $i=1, 2$. Hence
    \begin{equation*}
        \begin{split}
        |X(x_1)| \leq  \lambda^{-4k+2\omega(x_1)}\xi^{-4f(2x_1)+2\omega(x_1)} \sum_{\substack{x_1 < x_2 < x_1\frac{a+dL}{a}\\ x_2\equiv x_1\Mod d}}\left(\lambda^2\xi^2\right)^{\omega(x_2)}\left(\sum_{\substack{\frac{a}{x_1}\leq y < \frac{a+dL}{x_1}\\ y\equiv ax_1^{-1}\Mod d}}\left(\lambda^2\right)^{\omega(y)}\left(\xi^2\right)^{\omega(y, 2x_1)}\right)^2,
    \end{split}
    \end{equation*}
First we try to estimate the inner summation by Theorems~\ref{lem:Shiu} and \ref{thm:merten}. To see that Theorem~\ref{lem:Shiu} is applicable with $\alpha = \beta = 1/2$ and $A_1 = A_2 = 1$, note that we have $x' = \frac{a+dL}{x_1}$ and $y' = \frac{dL}{x_1}$, and clearly $y' > x'$. Since $x_1\le \sqrt{a+dL} \le \sqrt{2L\log L}$, we have $\frac{(y')^2}{x'} = \frac{d^2L^2}{x_1(a+dL)}\le \frac{d^2L^2}{(2L\log L)^{3/2}} > 1$ for $L$ sufficiently large. Meanwhile we have $\frac{y'}{d^2} = \frac{dL}{d^2x_1} \le \frac{dL}{d^2\sqrt{2L\log L}} > 1$ as $d < \sqrt{\log L}$ when $L$ sufficiently large. Note that $g(y) = \left(\lambda^2\right)^{\omega(y)}\left(\xi^2\right)^{\omega(y, 2x_1)}$ takes value $g(p) = \lambda^2$ for $2x_1 < p < \frac{a+dL}{x_1}$ and $g(p) = \lambda^2\xi^2$ for $p < 2x_1$. Therefore by Theorem~\ref{lem:Shiu},
\[\begin{split}\sum_{\substack{\frac{a}{x_1}\leq y < \frac{a+dL}{x_1}\\ y\equiv ax_1^{-1}\Mod d}}\left(\lambda^2\right)^{\omega(y)}\left(\xi^2\right)^{\omega(y, 2x_1)} &\ll \frac{\frac{dL}{x_1}}{\phi(d)}\frac{1}{\log\frac{a+dL}{x_1}}\exp\left(\sum_{p \leq 2x_1} \frac{\lambda^2\xi^2-\lambda^2}{p} + \sum_{p < \frac{a+dL}{x_1}}\frac{\lambda^2}{p}\right) \\
&\ll \frac{dL}{x_1\phi(d)}(\log(a+dL))^{-1+\lambda^2}(\log x_1)^{\lambda^2\xi^2-\lambda^2}.\end{split}\]
Here in the second line we use Theorem~\ref{thm:merten} to get $\sum_{p<\frac{a+dL}{x_1}}\frac{\lambda^2}{p} \leq \lambda^2\LL\frac{a+dL}{x_1}+O(1)$, and simiarly $\sum_{p<2x_1}\frac{\lambda^2\xi^2 - \lambda^2}{p} \leq (\lambda^2\xi^2 - \lambda^2)\LL 2x_1+O(1) = (\lambda^2\xi^2 - \lambda^2)\LL x_1 + O(1)$. 

For the outer summation, we argue that Theorem~\ref{lem:Shiu} is applicable when $x_1 > \frac{ea^2}{L^2}$. In this case we have $x' = x_1\frac{a+dL}{a} > y' = x_1\frac{dL}{a}$. Moreover $\frac{(y')^2}{x'} = \frac{x_1d^2L^2}{(a+dL)a} > \frac{x_1d^2L^2}{2a^2} > 1$ when $x_1 = \frac{2ea^2}{L^2}$. At the same time, we have $\frac{y'}{d^2} = \frac{x_1L}{ad} > \frac{ea}{dL} > 1$. Hence Theorem~\ref{lem:Shiu} is applicable with $\alpha=\beta=1/2$, so
\begin{equation}\label{eqn:prop6.2-outersum}
    \sum_{\substack{x_1 < x_2 < x_1\frac{a+dL}{a}\\ x_2\equiv x_1\Mod d}}\left(\lambda^2\xi^2\right)^{\omega(x_2)} \ll \frac{x_1\frac{dL}{a}}{\phi(d)}\left(\log \frac{x_1(a+dL)}{a}\right)^{-1+\lambda^2\xi^2} \leq \frac{x_1dL}{a\phi(d)}(\log x_1)^{-1+\lambda^2\xi^2}.
\end{equation}
Here in the second inequality we use Theorem~\ref{thm:merten} to get $\sum_{p<x_1\frac{a+dL}{a}}\frac{\lambda^2\xi^2}{p} \leq \lambda^2\xi^2\LL\frac{x_1(a+dL)}{a}+O(1)$. Let $X_1$ be the collection of tuples in $X$ with $x_1 < \frac{ea^2}{L^2}$, and $X_2 = X\setminus X_1$. Summing over $x_1 > ea^2/L^2$, noting that $\log L < \log(a+dL) \leq \log(2L^2) \leq 3\log L$ and that $\xi^{-4f(2x_1)} = \xi^{-4\beta}\xi^{-4\frac{\LL 2x_1}{\log 4}}\ll \xi^{-4\beta}(\L x_1)^{\frac{2\log 1/\xi}{\log 2}}$, we have
\[|X_2|\ll \lambda^{-4k}\xi^{-4\beta}\sum_{\frac{ea^2}{L^2} < x_1 < \sqrt{a+dL}}\frac{d^3L^3}{ax_1\phi(d)^3}(\log x_1)^{\frac{2\log 1/\xi}{\log 2}-1+3\lambda^2\xi^2-2\lambda^2}(\log L)^{-2+2\lambda^2}(\lambda^2\xi^2)^{\omega(x_1)}.\]
We partition integers $x_1$ in the interval $(\frac{ea^{2}}{L^{2}}, \sqrt{a+dL})$ into $x_1\in (e^t, e^{t+1})$ for $\lfloor\log\frac{ea^2}{L^2}\rfloor \leq t \leq \lfloor\log\sqrt{a+dL}\rfloor$. Note that by our choice $t\geq \lfloor\log\frac{ea^2}{L^2}\rfloor \geq 1$. In each interval we have
\[\begin{split}&\sum_{e^t < x_1 < e^{t+1}}\frac{d^3L^3}{ax_1\phi(d)^3}(\log x_1)^{\frac{2\log 1/\xi}{\log 2}-1+3\lambda^2\xi^2}(\log L)^{-2+2\lambda^2-2\lambda^2\xi^2}(\lambda^2\xi^2)^{\omega(x_1)}\\
[e^t<x_1<e^{t+1}] \ll & \frac{d^3L^3}{ae^t\phi(d)^3}t^{\frac{2\log 1/\xi}{\log 2}-1+3\lambda^2\xi^2-2\lambda^2}(\log L)^{-2+2\lambda^2}\sum_{e^t < x_1 < e^{t+1}}(\lambda^2\xi^2)^{\omega(x_1)}\\
[\mbox{by Theorem}~\ref{lem:Shiu}] \ll & \frac{d^3L^3}{ae^t\phi(d)^3}t^{\frac{2\log 1/\xi}{\log 2}-1+3\lambda^2\xi^2-2\lambda^2}(\log L)^{-2+2\lambda^2} \cdot e^t(e-1)(t+1)^{-1+\lambda^2\xi^2}\\
 \ll & \frac{d^3L^3}{a\phi(d)^3}t^{\frac{2\log 1/\xi}{\log 2}-2+4\lambda^2\xi^2-2\lambda^2}(\log L)^{-2+2\lambda^2}.
\end{split}\]
Optimizing the parameters by choosing $\lambda = \frac1{\sqrt{\L 4}}$ and $\xi = \frac1{\sqrt{2}}$, we have
\[\begin{split}|X_2| & \ll \lambda^{-4k}\xi^{-4\beta}\sum_{t = \floor{\log \frac{ea^2}{L^2}}}^{\floor{\log \sqrt{a+dL}}}\frac{d^3L^3}{a\phi(d)^3}t^{\frac{2\log 1/\xi}{\log 2}-2+4\lambda^2\xi^2-2\lambda^2}(\log L)^{-2+2\lambda^2}
\\
& = (\log 4)^{2k}2^{2\beta}\frac{d^3L^3}{a\phi(d)^3} (\log L)^{-2+\frac{1}{\log 2}}\sum_{t = \floor{\log \frac{ea^2}{L^2}}}^{\floor{\log \sqrt{a+dL}}}t^{-1} \\
& \ll L^{2}(\log L)^{-2+\frac{1}{\log 2}} \cdot  \log \log L\cdot (\log 4)^{2k}2^{2\beta},
\end{split}\]
where in the last step we use that $\frac{d^3L^3}{a\phi(d)^3}\leq \frac{d^2}{\phi(d)^3}L^2\ll L^2$. 

It remains to estimate $|X_1|$. Note that \eqref{eqn:prop6.2-outersum} no longer holds, so we replace it by the trivial bound
\[\sum_{\substack{x_1 < x_2 < x_1\frac{a+dL}{a}\\ x_2\equiv x_1\Mod d}}\left(\lambda^2\xi^2\right)^{\omega(x_2)} \le \sum_{\substack{x_1 < x_2 < x_1\frac{a+dL}{a}\\ x_2\equiv x_1\Mod d}}1 \le \frac{x_1L}{a}.\]
Also by the choice of $\lambda$ we have $\lambda^{-4k+2\omega(x_1)}\le \lambda^{-4k} = (\log 4)^{2k}$ and $\xi^{-4f(2x_1)}= \xi^{-4\beta}\xi^{-4\frac{\LL 2x_1}{\log 4}}\ll \xi^{-4\beta}(\log x_1)^{2\frac{\log 1/\xi}{\log 2}} = 2^{2\beta}\log x_1$. As a consequence, we have (again noting that $\log(a+dL) = \Theta(\log L)$)
\[\begin{split}|X_1| &\ll (\log 4)^{2k}2^{2\beta}\sum_{\frac{a}{L} < x_1 < \frac{ea^2}{L^2}}\frac{d^2L^3}{a\phi(d)^2x_1}(\log L)^{-2+2\lambda^2}(\log x_1)^{1 + 2\lambda^2\xi^2-2\lambda^2}\\
& = (\log 4)^{2k}2^{2\beta}L^2(\log L)^{-2+\frac{1}{\log 2}}\frac{d^2}{\phi(d)^2}\frac{L}{a}\sum_{\frac{a}{L}< x_1 < \frac{ea^2}{L^2}} \frac{(\log x_1)^{1 + \frac{1}{\log 4}}}{x_1}.\end{split}\]
Note that we have
\[\sum_{\frac{a}{L}< x_1 < \frac{ea^2}{L^2}} \frac{(\log x_1)^{1 + \frac{1}{\log 4}}}{x_1} \ll \left(\log \frac{ea^2}{L^2}\right)^{2+\frac{1}{\log 4}} \ll \sqrt{a/L}.\]
Moreover we have $\frac{d^2}{\phi(d)^2}\sqrt{\frac{L}{a}} \leq \frac{d^\frac32}{\phi(d)^2} \ll 1$. We conclude that $|X_1| \ll L^2(\log L)^{-2+\frac{1}{\log 2}}(\log 4)^{2k}2^{2\beta}$. Combining with the previous estimate on $|X_2|$, 
 we have the desired upper bound $|X| = |X_1|+|X_2| \ll L^{2}(\log L)^{-2+\frac{1}{\log 2}} \cdot  \log \log L\cdot (\log 4)^{2k}2^{2\beta}$.
\end{proof}

\section{Concluding remarks}\label{sec:concluding}

As we mentioned in the introduction, the best bound we can get in Theorem~\ref{cor:  AA 86} is the following.
\begin{theorem}\label{thm: best}
Let $\A\subseteq\Z$ be a finite arithmetic progression and $\theta = 1- \frac{1+\log \log 4}{\log 4}$. Then 
\[|\A \cdot \A| \geq \frac{ {|\mathcal A|^2}}{{(\log |\mathcal A|)^{2\theta} (\log \log |\A|)^{7+o(1)}  }}. \]
\end{theorem}
The proof of Theorem~\ref{thm: best} is identical to the proof of Theorem~\ref{cor:  AA 86} but we need a stronger estimate on the set $\CN_k(\A; \log 4, \beta)$ for $k$ with slightly larger size than the choice in Proposition~\ref{Prop: N}. The stronger estimate requires ideas from the generalized Smirnov statistics \cite{Fordstats07} which is deferred to Appendix~\ref{sec:Smirnov-stronger}.

We make the following conjecture on the sharp lower bound in Theorem~\ref{cor:  AA 86}.
\begin{conjecture}
Let $\A $ be a finite arithmetic progression in integers. Then for $|\A|$ sufficiently large,
\[|\A\cdot\A| \gg \frac{|\A|^2}{(\L |\A|)^{2\theta}(\LL |\A|)^{3/2}}.\]
\end{conjecture}
This conjecture, if true, would be tight up to a constant factor by considering $\A = [N]$. We believe that this cannot be done by only estimating the multiplicative energy of a subset.

We also extend Conjecture~\ref{conj: ER} to the $h$-fold product case.
\begin{conjecture}
Let $\A$ be a set of integers and $\A^h := \{a_1\cdots a_h:a_i\in \A,\;\forall\;1\leq i\le h\}$. If $|\A+\A|\ll |\A|$. Then
\[|\A^h| \geq |\A|^h(\log |\A|)^{-h\log h+h-1-o(1)}.\]
\end{conjecture}
One way to achieve this lower bound is by choosing $\A = \{1\leq n\leq N:\omega(n) = (1+o(1))\log\log N\}$.

\section*{Acknowledgments}
We would like to thank Yifan Jing and Cosmin Pohoata for helpful discussions about Conjecture~\ref{conj: ER}. We are indebted to Kevin Ford for his helpful comments and corrections to earlier versions of the paper. We are grateful to Kannan Soundararajan for discussions about Shiu's Theorem \cite{Shiu80} and work of Tenenbaum \cite{Ten84} and
Ford \cite{Ford08}. We thank Fernando Xuancheng Shao for pointing out reference \cite{EGM14} after reading an earlier draft of the paper. We appreciate Huy Tuan Pham for his interest in and discussions on the problem. We thank Dimitris Koukoulopoulos for pointing out a mistake in the reference and Marc Munsch for pointing out a missing reference. We also thank the anonymous referee for valuable comments and feedback. Finally, we are indebted to Jacob Fox for encouraging us to write this paper and for carefully reading through earlier drafts of the paper and for useful suggestions.

	\bibliographystyle{amsplain}
	\bibliography{main}{}

\providecommand{\bysame}{\leavevmode\hbox to3em{\hrulefill}\thinspace}
\providecommand{\MR}{\relax\ifhmode\unskip\space\fi MR }
\providecommand{\MRhref}[2]{%
  \href{http://www.ams.org/mathscinet-getitem?mr=#1}{#2}
}
\providecommand{\href}[2]{#2}
\begin{thebibliography}{10}

\bibitem{camp20}
M.~Campos, \emph{On the number of sets with a given doubling constant}, Israel
  J. Math. \textbf{236} (2020), no.~2, 711--726. \MR{4093901}

\bibitem{campos2021typical}
M.~Campos, M.~Coulson, O.~Serra, and M.~W\"{o}tzel, \emph{The typical
  approximate structure of sets with bounded sumset}, 2021.

\bibitem{Chang-unpub}
M.-C. Chang, \emph{Product sets of arithmetic progressions}, Unpublished
  manuscript.

\bibitem{Chang}
M.-C. Chang, \emph{The {E}rd{\H{o}}s-{S}zem{\'e}redi problem on sum set and
  product set}, Ann. of Math. (2) \textbf{157} (2003), no.~3, 939--957.
  \MR{1983786}

\bibitem{dMT}
R.~{de la Bret{\`e}che}, M.~{Munsch}, and G.~{Tenenbaum}, \emph{{Small G{\'a}l
  sums and applications}}, arXiv e-prints (2019), arXiv:1906.12203.

\bibitem{EGM14}
S.~Eberhard, B.~Green, and F.~Manners, \emph{Sets of integers with no large
  sum-free subset}, Ann. of Math. (2) \textbf{180} (2014), no.~2, 621--652.
  \MR{3224720}

\bibitem{ElekesRuzsa}
Gy. Elekes and I.~Z. Ruzsa, \emph{Few sums, many products}, Studia Sci. Math.
  Hungar. \textbf{40} (2003), no.~3, 301--308. \MR{2036961}

\bibitem{Erd60}
P.~Erd\H{o}s, \emph{An asymptotic inequality in the theory of numbers}, Vestnik
  Leningrad. Univ. \textbf{15} (1960), no.~13, 41--49. \MR{0126424}

\bibitem{ES83}
P.~Erd\H{o}s and E.~Szemer\'{e}di, \emph{On sums and products of integers},
  Studies in pure mathematics, Birkh\"{a}user, Basel, 1983, pp.~213--218.
  \MR{820223}

\bibitem{Erdos55}
P.~Erd{\H{o}}s, \emph{Some remarks on number theory}, Riveon Lematematika
  \textbf{9} (1955), 45--48. \MR{73619}

\bibitem{Fordstats07}
K.~Ford, \emph{Generalized {S}mirnov statistics and the distribution of prime
  factors}, Funct. Approx. Comment. Math. \textbf{37} (2007), no.~part 1,
  119--129. \MR{2357313}

\bibitem{Ford08}
\bysame, \emph{The distribution of integers with a divisor in a given
  interval}, Ann. of Math. (2) \textbf{168} (2008), no.~2, 367--433.
  \MR{2434882}

\bibitem{ford2008sharp}
\bysame, \emph{Sharp probability estimates for generalized {S}mirnov
  statistics}, Monatshefte fur Mathematik \textbf{153} (2008), no.~3, 205--216.

\bibitem{FordExtremal18}
\bysame, \emph{Extremal properties of product sets}, Proc. Steklov Inst. Math.
  \textbf{303} (2018), no.~1, 220--226, Published in Russian in Tr. Mat. Inst.
  Steklova {{\bf{3}}03} (2018), 239--245. \MR{3920222}

\bibitem{Freiman73}
G.~A. Fre\u{\i}man, \emph{Foundations of a structural theory of set addition},
  Translations of Mathematical Monographs, Vol 37, American Mathematical
  Society, Providence, R.I., 1973, Translated from the Russian. \MR{0360496}

\bibitem{Kou14}
D.~Koukoulopoulos, \emph{On the number of integers in a generalized
  multiplication table}, J. Reine Angew. Math. \textbf{689} (2014), 33--99.
  \MR{3187928}

\bibitem{Lev}
V.~F. Lev, \emph{Restricted set addition in groups. {III}. {I}nteger sumsets
  with generic restrictions}, Period. Math. Hungar. \textbf{42} (2001),
  no.~1-2, 89--98. \MR{1832697}

\bibitem{Ma}
D.~Mastrostefano, \emph{On maximal product sets of random sets}, J. Number
  Theory \textbf{224} (2021), 13--40. \MR{4221525}

\bibitem{MRSS}
B.~Murphy, M.~Rudnev, I.~Shkredov, and Y.~Shteinikov, \emph{On the few
  products, many sums problem}, J. Th\'{e}or. Nombres Bordeaux \textbf{31}
  (2019), no.~3, 573--602. \MR{4102615}

\bibitem{NT99}
M.~B. Nathanson and G.~Tenenbaum, \emph{Inverse theorems and the number of sums
  and products}, Ast\'{e}risque \textbf{258} (1999), xiii, 195--204, Structure
  theory of set addition. \MR{1701198}

\bibitem{PS90}
C.~Pomerance and A.~S\'{a}rk\"{o}zy, \emph{On products of sequences of
  integers}, Number theory, {V}ol. {I} ({B}udapest, 1987), Colloq. Math. Soc.
  J\'{a}nos Bolyai, vol.~51, North-Holland, Amsterdam, 1990, pp.~447--463.
  \MR{1058228}

\bibitem{RSS2020}
M.~Rudnev, G.~Shakan, and I.~D. Shkredov, \emph{Stronger sum-product
  inequalities for small sets}, Proc. Amer. Math. Soc. \textbf{148} (2020),
  no.~4, 1467--1479. \MR{4069186}

\bibitem{RS20}
M.~Rudnev and S.~Stevens, \emph{An update on the sum-product problem}, 2021.

\bibitem{Ruzsa94}
I.~Z. Ruzsa, \emph{Generalized arithmetical progressions and sumsets}, Acta
  Math. Hungar. \textbf{65} (1994), no.~4, 379--388. \MR{1281447}

\bibitem{Selberg}
A.~Selberg, \emph{Note on a paper by {L}. {G}. {S}athe}, J. Indian Math. Soc.
  (N.S.) \textbf{18} (1954), 83--87. \MR{67143}

\bibitem{SX19}
X.~Shao and W.~Xu, \emph{A robust version of {F}reiman's {$3k-4$} theorem and
  applications}, Math. Proc. Cambridge Philos. Soc. \textbf{166} (2019), no.~3,
  567--581. \MR{3933910}

\bibitem{Shiu80}
P.~Shiu, \emph{A {B}run-{T}itchmarsh theorem for multiplicative functions}, J.
  Reine Angew. Math. \textbf{313} (1980), 161--170. \MR{552470}

\bibitem{Solymosi09}
J.~Solymosi, \emph{Bounding multiplicative energy by the sumset}, Adv. Math.
  \textbf{222} (2009), no.~2, 402--408. \MR{2538014}

\bibitem{SoundXu}
K.~{Soundararajan} and M.~W. {Xu}, \emph{{Central limit theorems for random
  multiplicative functions}}, arXiv e-prints (2022), arXiv:2212.06098.

\bibitem{TaoVu}
T.~Tao and V.~Vu, \emph{Additive combinatorics}, Cambridge Studies in Advanced
  Mathematics, vol. 105, Cambridge University Press, Cambridge, 2006.
  \MR{2289012}

\bibitem{Ten84}
G.~Tenenbaum, \emph{Sur la probabilit\'{e} qu'un entier poss\`ede un diviseur
  dans un intervalle donn\'{e}}, Compositio Math. \textbf{51} (1984), no.~2,
  243--263. \MR{739737}

\bibitem{Tenenbaum}
\bysame, \emph{Introduction to analytic and probabilistic number theory}, third
  ed., Graduate Studies in Mathematics, vol. 163, American Mathematical
  Society, Providence, RI, 2015, Translated from the 2008 French edition by
  Patrick D. F. Ion. \MR{3363366}

\end{thebibliography}

\appendix	
	\section{Number of elements with a fixed number of primes factors in arithmetic progressions via generalized Smirnov statistics} \label{sec:Smirnov-stronger}
In this section, we give a stronger estimate on size of
$\CN_k(\A; \alpha, \beta)$, than the bound in Proposition ~\ref{Prop: N} (Section~\ref{sec: Smirnov}). This is the key ingredient in proving Theorem~\ref{thm: best}. Notably, our choice of $k$ here (see Proposition~\ref{A Prop: N}) is slightly larger than the one we used in Proposition~\ref{Prop: N}. The rest proof strategy of Theorem~\ref{thm: best} is identical to the proof of Theorem~\ref{cor:  AA 86}. By applying Proposition~\ref{thm: energy bound} and Lemma~\ref{lem:random-selection} to give corresponding upper bounds on multiplicative energy, an application of Cauchy-Schwartz inequality gives lower bounds on product sets. 

Recall that the set we are interested in Section~\ref{sec: Smirnov} is the following:
\begin{equation}\label{A eqn: N}
    \CN_k(\A; \alpha, \beta) := \{n \in \A: n\textrm{  square-free},\; \omega(n) = k,\; \LL p_j(n) \geq \alpha j-\beta\;~~ \forall\;1\leq j\leq k\}.
\end{equation}
Let $N_k(\A; \alpha, \beta) = |\CN_k(\A; \alpha, \beta)|$.
Our goal is to give a better lower bound on $N_k(\A; \alpha, \beta)$, comparing to Proposition \ref{Prop: N}. The proof uses ideas in \cite{Fordstats07}.

\begin{proposition}\label{A Prop: N}
There exists an absolute constant $\gamma$ such that the following holds. If $\A = \{a+id:0\leq i < L\}$ is an arithmetic progression satisfying that $L$ is sufficiently large and $0 < dL \leq a \leq L\sqrt{\log L}$, and $\gcd(a, d) = 1$, then 
\[N_k(\A; \log 4, \gamma \LLLL L) 
\gg \frac{L}{(\log L)^{\theta}} (\L\L L)^{-3/2} (\L\L\L L)^{-C_0}
\] for $k = \left\lfloor\frac{\LL L}{\L 4}\right\rfloor$ and some fixed $C_0>0$.
\end{proposition}

\begin{proof}[Proof of Proposition~\ref{A Prop: N} assuming Lemma~\ref{A lem:sum-of-reciprocal}]
Let $\beta = \gamma \LLLL L$, and we may write \newline $\CN_k = \CN_k(\A; \log 4, \beta)$ as the inputs are clear from the context. To obtain a lower bound, we consider elements $n\in \CN_k$ with $p_1(n)p_2(n)\cdots p_{k-1}(n) < \sqrt{a}$. 

Once we fix such a choice of $(p_1, p_2, \dots, p_{k-1})$ with $p_1\cdots p_{k-1} < \sqrt{a}$ and $\LL p_j \geq j\log 4 - \beta$ for all $1\le j \le k-1$, we count the number of $p_k$ with $p_k > p_{k-1}$ and $p_1p_2\cdots p_k\in \A$. Let $q = p_1p_2\cdots p_{k-1}$. Clearly if $qp\in \A$ for some prime $p$, then $p \geq \frac{a}{q} > \sqrt{a} > q > p_{k-1}$. Moreover we have $\LL p > \LL \sqrt{L} = \LL L - \L 2 > \alpha k - \beta$ when $\beta$ is sufficiently large, so any such choice of prime $p$ would give a valid choice for $p_k$ as defined in \eqref{A eqn: N}. 

For any such fixed choice of $q$, if $\gcd(q, d) > 1$, then any element in $\A$ is not a multiple of $q$, so there is no such choice of $p_k$. 

If $\gcd(q, d) = 1$, we know that $q|a+id$ is equivalent to $i\equiv -ad^{-1}\pmod q$. There is a unique such $i_0$ with $0\leq i_0 < q$. Then the multiples of $q$ in $\A$ are of the form $$\left\{a+(i_0+k'q)d:0\leq k' < \frac{L-i_0}{q}\right\} \supseteq \{a+(i_0+k'q)d:0\leq k' < \lfloor L/q\rfloor\}.$$ The inclusion is because $\frac{L-i_0}{q} > \lfloor L/q\rfloor - 1$.
Hence any element $p\in \A_q := \{a'+k'd:0\leq k' < L'\}$ where $a' = \frac{a+i_0d}{q}$ and $L' = \lfloor L/q\rfloor$ would satisfy $pq\in \A$. It is sufficient to estimate the number of primes from $\A_q$, which itself is an arithmetic progression. As $\gcd(a', d) \leq \gcd(a+i_0d, d) = \gcd(a, d) = 1$, we have $\gcd(a', d) = 1$. Moreover, we know that $dL' \leq \frac{dL}{q} < \frac{a}{q} \leq a'$ and $a' < 10L'\sqrt{\log L'}$ when $L$ is sufficiently large. We also have that $L' \geq \frac{L}{q} - 1 > \frac{L^{\frac12}}{(\L L)^{\frac14}} - 1$ is also sufficiently large. Hence by Lemma~\ref{lem:primes-in-AP}, the number of choices for $p_k$ is at least
\begin{equation}\label{A eqn:  count-last-prime}
    \pi(\A_q) \geq \frac{dL'}{2\phi(d)\L L'} \geq \frac{dL/q}{4\phi(d)\L L}.
\end{equation}
Hence by summing \eqref{A eqn:  count-last-prime} up, we have
\begin{equation}\label{A eqn:  prop-N-target}
    N_k(\A; \log 4, \beta) \geq \frac{dL}{4\phi(d)\log L}\cdot \sum_{\substack{p_1 < \cdots < p_{k-1}, p_1\cdots p_{k-1} < \sqrt{a}\\ p_j\nmid d,\LL p_j \geq j\log 4 - \beta\;\forall\;1\leq j\leq k-1}}\frac{1}{p_1\cdots p_{k-1}}.
\end{equation}
To estimate the sum on the right hand side of \eqref{A eqn: prop-N-target}, we use Lemma~\ref{A lem:sum-of-reciprocal}. Let $\gamma', C'$ be the constants in Lemma~\ref{A lem:sum-of-reciprocal}, and let $(x', k', d', \beta') = (\sqrt{a}, k-1, d, \beta)$. We verify that the assumptions are met. First, as $x' \geq \sqrt{a} \geq \sqrt{L}$, we may assume that $x'$ is sufficiently large by setting $L$ large enough. As $L$ is sufficiently large, we have $L^2 \geq a \geq L$, so $L \geq x' \geq \sqrt{L}$. Hence $\LL x' = \LL L + O(1)$, and
\[k' = k-1 = \frac{\LL L }{\L 4}+O(1) = \frac{\LL x'}{\L 4} + O(1).\]
Also we have $d \leq 10\sqrt{\L L} \leq 10\sqrt{\L (x')^2} = O(\L x')$. By choosing $\gamma = \gamma'$, we have $\beta' = \beta = \gamma\LLLL L \geq \gamma \LLLL x'$. Therefore we may apply Lemma~\ref{A lem:sum-of-reciprocal} with $a$ being polynomial in $L$ to get that there exists a constant $C_0>0$ such that, by recalling the definition of $\theta$,
\[\begin{split}
N_k(\A; \log 4, \beta) & \gg \frac{dL}{\phi(d)\log L}\cdot(\log L)^{\frac{1+ \L\L 4}{\L 4}} (\L\L L)^{-3/2} (\L\L\L L)^{-C_0}\\
& \gg \frac{L}{(\log L)^{\theta}} (\L\L L)^{-3/2} (\L\L\L L)^{-C_0},  
\end{split}
\]
where in the last inequality we use that $\frac{d}{\phi(d)}\geq 1$.
\end{proof}

\begin{theorem}[Theorem 1 in \cite{ford2008sharp}]\label{thm:Smirnov}
Let $\{U_i\}_{i=1}^n$ be independent uniform random variables on $[0, 1]$ and $\mathbf{1}_E$ be the indicator function for event $E$. Uniformly for $n\in \N$ and $u, w > 0$, we have
\begin{equation}\label{eqn:Smirnov-original}
    \Pr\left(\sum_{i=1}^n\mathbf{1}_{U_i \leq t} \leq (n+w-u)t+u\;\forall\; t\in [0, 1]\right) = 1 - e^{-\frac{2uw}{n}} + O\left(\frac{u+w}{n}\right).
\end{equation}
\end{theorem}
\begin{corollary}\label{cor:Smirnov}There exists a sufficiently large constant $C$ such that the following holds. Uniformly for $\beta \geq C\alpha > 0$ and $C \leq \frac{N-\alpha n+\beta}{\alpha} $, we have
        \[\vol\left\{(x_1, \dots, x_n):\begin{matrix}0\leq x_1\leq \cdots \leq x_n\leq N,\\ x_j\geq \alpha j -\beta \;\forall\; 1\leq j\leq n\end{matrix}\right\} \le 3\frac{N^n}{n!}\frac{\beta(N-\alpha n+\beta)}{n\alpha^2}.\]
Moreover, if we further have that $\frac{\beta(N-\alpha n+\beta)}{n\alpha^2}\leq 1$, then
\[\vol\left\{(x_1, \dots, x_n):\begin{matrix}0\leq x_1\leq \cdots \leq x_n\leq N,\\ x_j\geq \alpha j -\beta \;\forall\; 1\leq j\leq n\end{matrix}\right\} \ge \frac{1}{4}\frac{N^n}{n!}\frac{\beta(N-\alpha n+\beta)}{n\alpha^2}.\]
\end{corollary}
\begin{proof}
    We may shrink the region by $N$ along each dimension. Now $(N, \alpha, \beta) = (1, \frac{\alpha}{N}, \frac{\beta}{N})$. As the volume of $\{x_i = x_j\}$ is zero, we may assume that they are all distinct without affecting the volume. Then we know that the volume of the desired region is $\frac{1}{n!}$ of the volume of the following region:
    \[R' := \{x_1, \dots, x_n\in [0, 1]^{n}: \textrm{the $j$-th smallest element is at least $(\alpha j - \beta)/N$ for all $1\leq j\leq n$}\}.\]
    As $\vol\{x_j = (\alpha j-\beta)/N\} = 0$ for each $j$, we can rewrite $R'$ as
    \[R'' := \left\{x_1, \dots, x_n\in [0, 1]^{n}: \textrm{the number of elements $\leq t$ is at most $\frac{Nt+\beta}{\alpha}$ for all $t\in [0, 1]$}\right\}\]
    without affecting the volume, i.e. $\vol(R'') = \vol(R')$. Compare it with Theorem~\ref{thm:Smirnov}. Choose $u = \frac{\beta}{\alpha} \geq C$ and $w = \frac{N-\alpha n+\beta}{\alpha} \in [C, n/u]$. We denote the left hand side of \eqref{eqn:Smirnov-original} by $Q_n(u, w)$. We have $\vol(R') = \vol(R'') = Q_n(u, w)$. By Theorem~\ref{cor:Smirnov}, there exists an absolute constant $C_0$ such that
    \[\left|Q_n(u, w) - \left(1-e^{-\frac{2uw}{n}}\right)\right| \leq C_0\frac{u+w}{n}.\]
    Pick $C > 8C_0$. Then $C_0\frac{u+w}{n} \leq \frac{1}{4}\frac{uw}{n}$. Since $1-x\leq e^{-x}$, we conclude that
    \[\vol(R') = Q_n(u, w)\leq 1-e^{-\frac{2uw}{n}}+C_0\frac{u+w}{n} \leq \frac{2uw}{n}+\frac{1}{4}\frac{uw}{n} <3 \frac{uw}{n} = 3\frac{\beta(N-\alpha n+\beta)}{n\alpha^2}.\]
    
    With the extra condition that $\frac{\beta(N-\alpha n+\beta)}{n\alpha^2}\leq $, we have $\frac{uw}{n}\leq 1$. Also note that for $x\in [0, 2]$, $e^{-x}\leq 1-\frac{x}{4}$. Therefore we conclude that, as $\frac{2uw}{n}\in [0, 2]$,
    \[\vol(R') =Q_n(u, w) \geq \frac{uw}{2n} - C_0\frac{u+w}{n} \geq \frac{uw}{4n} = \frac{1}{4}\frac{\beta(N-\alpha n+\beta)}{n\alpha^2}.\]
    Putting back the factor $\frac{N^n}{n!}$, we get the desired statements.
\end{proof}

\begin{lemma}\label{A lem:sum-of-reciprocal}
There exist constants $\gamma$ and $C_0>0$ so that for any $x$ large enough, $k = \frac{\LL x}{\log 4} + O(1)$, $d = O(\L x)$, and $\beta \geq \gamma\LLLL x$, we have
\begin{equation}\label{A eqn:sum-of-reciprocal}
    \sum_{\substack{p_1 < \cdots < p_{k}, p_1\cdots p_{k} < x\\ p_j\nmid d,\LL p_j \geq j\log 4 - \beta\;\forall\;1\leq j\leq k}}\frac{1}{p_1\cdots p_{k}} \ge (\log x)^{\frac{1+ \L\L 4}{\L 4}} (\L\L x)^{-3/2} (\L\L\L x)^{-C_0}
\end{equation}
\end{lemma}

First let us explain the intuition behind the right hand side of \eqref{A eqn:sum-of-reciprocal}. If we replace the conditions by that $p_j < x$ and $p_j\nmid d$ for all $1\leq j\leq k$, then the sum would be
\[\left(\sum_{p < x, p\nmid d}\frac{1}{p}\right)^k = (\LL x - O(\LLL d))^k = (\LL x)^k(\LL d)^{-O(1)} = (\LL x)^{k-o(1)}.\]
Here we repeat each term on the left hand side of \eqref{A eqn:sum-of-reciprocal} by exactly $k!$ times. Thus we should expect a formula of the form $\frac{(\LL x)^{k-o(1)}}{k!}$. Now we aim to show that, with these extra conditions on the choices of $(p_1, \dots, p_k)$, the sum stays roughly the same.
\begin{proof}Motivated by \cite{Fordstats07}, we partition primes in $\P = \{p:p\nmid d\}$ into consecutive parts. Fix $\lambda_0 = 1.9$, and for $j\geq 1$ let $\lambda_j$ to be the largest prime such that
\[\sum_{p \leq \lambda_j:p\nmid d}\frac{1}{p}\leq j.\]
Let $\P_j = \{p: \lambda_{j-1} < p \leq \lambda_j, p\nmid d\}$. First we can estimate the size of $\lambda_j$. By maximality of $\lambda_j$ we have
\begin{equation}\label{eqn:lambda_j-lower}
    j-1 < j - \frac{1}{\lambda_j} < \sum_{p\leq \lambda_j:p\nmid d}\frac{1}{p} \leq \sum_{p\leq \lambda_j}\frac{1}{p}
\end{equation}
Meanwhile note that $\sum_{p|d}\frac{1}{p} = O(\LL\omega(d)) = O(\LLLL x)$, so we have
\[\sum_{p\leq \lambda_j}\frac{1}{p} \leq \sum_{p\leq \lambda_j:p\nmid d}\frac{1}{p} + \sum_{p|d}\frac{1}{p} \leq j + O(\LLLL x).\]
By Merten's estimate (e.g. see Theorem 1.10 in \cite{Tenenbaum}), we have
\[\sum_{p\leq \lambda_j}\frac{1}{p} = \LL\lambda_j + O(1).\]
Thus we conclude that for some absolute constant $C > 0$,
\[j - C < \LL \lambda_j < j + C\LLLL x.\]
Moreover, we know that by \eqref{eqn:lambda_j-lower} we have
\begin{equation}\label{eqn:P_j-lower}
    \sum_{p\in \P_j}\frac{1}{p} = \sum_{p\leq \lambda_j: p\nmid d}\frac{1}{p} - \sum_{p\leq \lambda_{j-1}:p\nmid d}\frac{1}{p} \geq j-\frac{1}{\lambda_j} - (j-1) = 1-\frac{1}{\lambda_j}.
\end{equation}
 Fix positive integers $t_1\leq t_2\leq  \dots\leq t_k$.

\[\P(t_1, \dots, t_k) := \{(p_1, \dots, p_k) \in \P^{k}: p_1 < \cdots< p_k, \quad p_j \in \P_{t_j}, \quad \forall 1\le j\le k\}.\]
We would like to reduce the summation in \eqref{A eqn:sum-of-reciprocal} over $(p_1, \dots, p_k)$ into summation over $(t_1, \dots, t_k)$. Thus we would like to impose some restrictions on $(t_1, \dots, t_k)$ so that the conditions $p_1\cdots p_k < x$, $p_j\nmid d$, and $\LL p_j \geq j\log 4 - \beta$ are all satisfied. Clearly $p_j\nmid d$ is satisfied for all $j$. Note that we have
\[\LL p_j \geq \LL \lambda_{t_j-1} \geq t_j-1-C,\]
so the condition $\LL p_j\geq j\log 4 - \beta$ is satisfied if we require that $t_j\geq j\log 4 - (\beta - 1 - C)$. Meanwhile, we also need $p_1\cdots p_k < x$. Note that we have
\[\sum_{j=1}^k \L p_j \leq \sum_{j = 1}^k \L \lambda_{t_j} \leq \sum_{j = 1}^k e^{t_j+C\LLLL x}.\]
Hence if we define $M = \LL x - C\LLLL x$, then $p_1\cdots p_k < x$ if
\begin{equation}\label{eqn:t_j-up}
    \sum_{j=1}^k e^{t_j} \leq e^{M}.
\end{equation}
Now we impose the restriction that $t_j$'s are not concentrated at small values or large values: if many $t_j$'s are small, then it is hard to take distinct $p_j$'s from the same $\P_t$; if many of them are large then \eqref{eqn:t_j-up} is disobeyed. Let $T$ be a sufficiently large constant to be determined. We choose our $t_j$ so that they only take values in $[T+1, M-T]$, and $t_{mT+1} \geq T+m+1$ and $t_{k-mT} \leq M-T-m$ for all positive integers $m < \frac{k}{T}$. Let $b_t$ be the number of $1\leq j\leq k$ with $t_j = T+t$ for $1 \leq t \leq M-2T$. As $t_1 \leq \cdots \leq t_k$, the numbers $t_1, \dots, t_k$ are uniquely determined by $b_1,\dots, b_{M-2T}$. Moreover, by our previous restriction, we know that $b_t\leq \min(Tt, T(M-2T-t+1), k)$ for all $1\leq t\leq M-2T$. First we argue that by choosing $T$ sufficiently large, \eqref{eqn:t_j-up} is satisfied. Indeed,
\begin{equation}
    \sum_{j=1}^k e^{t_j} = \sum_{t=1}^{M-2T}b_te^{T+t} \leq \sum_{t=1}^{M-2T}T(M-2T-t+1)e^{T+t}
     = \sum_{s=1}^{M-2T}Tse^{M-T+1-s}
     < Te^{M-T+1}
\end{equation}
which is less than $e^M$ if we take $T > 2$. Now we have
\begin{equation}\label{eqn:fix-t}
\begin{split}\sum_{(p_1, \dots,  p_k)\in \P(t_1, \dots, t_k)}\frac{1}{p_1 \cdots  p_k} &= \prod_{t=1}^{M-2T}\frac{1}{b_t!}\left(\sum_{p_1\in \P_{t+T}}\frac{1}{p_1}\sum_{\substack{p_2\in \P_{t+T}\\p_2\ne p_1}}\frac{1}{p_2}\cdots \sum_{\substack{p_{b_t}\in \P_{t+T}\\
p_{b_t}\ne p_1, \dots, p_{b_t-1}}}\frac{1}{p_{b_t}}\right)\\
& \geq \prod_{t=1}^{M-2T}\frac{1}{b_t!} \left(\sum_{p\in \P_{t+T}}\frac{1}{p} - \frac{b_t-1}{\lambda_{t+T-1}}\right)^{b_t}\\
[\mbox{by } \eqref{eqn:P_j-lower}] \quad & \geq \prod_{t=1}^{M-2T}\frac{1}{b_t!} \left(1 - \frac{b_t}{\lambda_{t+T-1}}\right)^{b_t}\\
[\mbox{by } b_t \leq tT]\quad & \geq \prod_{t=1}^{M-2T}\frac{1}{b_t!}\left(1-\frac{tT}{\exp\exp(t+T-1-C)}\right)^{tT}\\
& \geq \frac{1}{2}\prod_{t=1}^{M-2T}\frac{1}{b_t!}
\end{split}
\end{equation}
where the last inequality is derived by picking $T$ sufficiently large. Omitting the coefficient $1/2$, \eqref{eqn:fix-t} is the volume of the region
\[R(t_1, \dots, t_k) := \left\{(x_1, \dots, x_k)\in \R^k:\begin{matrix}0 \leq x_1 \leq x_2 \leq \cdots \leq x_k\leq M-2T,\\ t_j-T-1 \leq x_j \leq t_j-T\;\forall\;1\leq j\leq k\end{matrix}\right\}.\]
Recall that $t_j$ satisfies the following conditions: $t_j \geq j\log 4 - (\beta - 1 - C)$ for any $1\leq j\leq k$, and $t_{k-mT}\geq T+m+1$ and $t_{k-mT} \leq M-T-m$ for all positive integers $m < \frac{k}{T}$. Therefore we conclude that the union of $R(t_1, \dots, t_k)$ over these choices of $(t_1, \dots, t_k)$ contains the following region
\[R := \left\{(x_1, \dots, x_k)\in \R^k:\begin{matrix} 0\leq x_1 \leq x_2 \leq \cdots \leq x_k\leq M-2T, \\
x_j\geq j\log 4 - (\beta-1-C)\;\forall\; 1\leq j\leq k,\\
x_{mT+1} \geq T+m, x_{k-mT} \leq M-T-m\;\forall\; 1\leq m < k/T\end{matrix}\right\}.\]
For simplicity let $\beta' = \beta - 1 - C$. Let us define
\[R_0 := \left\{(x_1, \dots, x_k)\in \R^k: 0\leq x_1 \leq x_2 \leq \cdots \leq x_k\leq M-2T, 
x_j\geq j\log 4 - \beta'\;\forall\; 1\leq j\leq k\right\},\]
and for $1\leq m < \frac{k}{T}$,
\[V_1(m) = \vol(R_0\cap \{x_{mT+1} < m\});\quad V_2(m) = \vol(R_0\cap \{x_{k-mT} > M-2T-m\}).\]
Then we have $R = R_0\setminus \left(\bigcup_{1\leq m < k/T}V_1(m)\cup V_2(m)\right)$, so
\begin{equation}\label{eqn: vol R}
    \vol(R) \geq \vol(R_0)- \sum_{1\leq m < k/T} V_1(m) - \sum_{1\leq m < k/T}V_2(m).
\end{equation}
First we estimate $\vol(R_0)$. By Corollary~\ref{cor:Smirnov}, we may set $(\alpha, \beta, N, n) = (\log 4, \beta', M-2T, k)$. When $\gamma > C$ and $x$ is sufficiently large, we have $\beta' > \log 4 > 0$ and
\begin{equation}\label{eqn: M-2T}
    (M-2T)- k\log 4 + \beta' = (\gamma - C)\LLLL x + O(1).
\end{equation}
 Hence the assumptions of Corollary~\ref{cor:Smirnov} are met, so
\[\vol(R_0) \gg \frac{(M-2T)^k}{k!}\frac{(\LLLL x)^2}{k}.\]
For each $V_1(m)$, we notice that the first $mT+1$ coordinates contribute volume at most $\frac{m^{mT+1}}{(mT+1)!}$, so
\[  V_1(m)  \leq \frac{m^{mT+1}}{(mT+1)!} \cdot \vol\{ 0\le x_{mT+2}\le \cdots \le x_{k} \le M-2T : x_j \ge \alpha j - \beta' ~~\forall~ mT+2 \le j \le k   \}.\]
By changing the variables, we shall rewrite as $V_1(m) \le  \frac{m^{mT+1}}{(mT+1)!} S_1(m)$ where
\[S_1(m): =  \vol\left\{ \begin{matrix}0\le y_1\le y_2\le \cdots \le y_{k-mT-1} \le M-2T :\\
y_j \ge (j+mT+1)\log 4 - \beta' ~~\forall~ 1 \le j \le k-mT-1  \end{matrix} \right\} . \]
Similarly for $V_2(m)$ we have $V_2(m) \leq \frac{m^{mT+1}}{(mT+1)!} S_2(m)$ where
\[S_2(m): =  \vol\left\{ \begin{matrix}0\le y_1\le y_2\le \cdots \le y_{k-mT-1} \le M-2T :
y_j \ge j\log 4 - \beta' ~~\forall~ 1 \le j \le k-mT-1  \end{matrix} \right\}. \]
By comparing the definition of $S_1(m)$ and $S_2(m)$, we see that $(j+mT+1)\log 4-\beta' > j\log 4 - \beta'$, so the region of $S_1(m)$ is a subset of that of $S_2(m)$. Therefore $S_1(m)\leq S_2(m)$. We next give an upper bound on $S_2(m)$.
Since $\beta \geq \gamma\LLLL x$, the  conditions in Corollary~\ref{cor:Smirnov} is satisfied once $x$ is sufficiently large (which is true as long as $L$ is large enough). By applying Corollary~\ref{cor:Smirnov} 
we have that  
\[ S_2(m) \le  \frac{(M-2T)^{k-1-mT}}{(k-1-mT)!} \cdot \frac{\beta'(M-2T- \log 4 (k-1-mT)+\beta')}{(k-1-mT)\alpha^2}.  \]
By using \eqref{eqn: M-2T}, the upper bound can be simplified as 
\[  \le  \frac{(M-2T)^{k-1-mT}}{(k-1-mT)!} \cdot \frac{\beta' (\gamma-C)\L\L\L\L x  + \beta'(mT+O(1))}{k-1-mT}.  \]
We next use the binomial identity $\binom{k}{mT+1} = \frac{k!}{(mT+1)! (k-1-mT)!}$ to get upper bound
\[ \frac{(M-2T)^{k}}{k!} \cdot  \frac{(mT+1)!}{(M-2T)^{mT+1}}  \cdot \binom{k}{mT+1}\cdot \frac{(\L\L\L\L x + mT)\L\L\L\L x}{k-1-mT},\]
and now we proceed by estimating binomial coefficients to get 
\[S_2(m) \le  \frac{(M-2T)^{k}}{k!} \cdot (mT+1)! \left( \frac{k}{e(M-2T)mT} \right)^{mT+1} \cdot (\L\L\L\L x +mT) \L\L\L\L x .\]
For $i = 1, 2$, as $V_i(m) \leq \frac{m^{mT+1}}{(mT+1)!}S_i(m)\leq \frac{m^{mT+1}}{(mT+1)!}S_2(m)$, we have
\[\begin{split}
    V_i(m) & \le  \frac{(M-2T)^{k}}{k!} \cdot  \left( \frac{k}{e(M-2T)T} \right)^{mT+1} \cdot (\L\L\L\L x +mT) \L\L\L\L x \\
   & \le \frac{(M-2T)^{k}}{k!} \cdot \frac{1}{T^{mT+1}} \cdot (\L\L\L\L x )^{2} (mT)^{2} .  
\end{split}  \]
We next sum up all $V_i(m)$ over $m<k/T$, and compare the sum with $\vol(R_0)$. One has
\[ 
\frac{1}{\vol(R_0)}\sum_{1\leq m < k/T} V_i(m)  \le \sum_{ m \ge 1}   \frac{(mT)^{2}}{T^{mT+1}}  \leq \frac{1}{2^T}
\]
for any $T \geq 8$. Thus, by choosing $T =8$, we conclude that for $i = 1, 2$,
\[ \sum_{m \le k/T} V_i(m) < \frac{1}{4} \vol (R_0). \]
Plugging these two bounds in \eqref{eqn: vol R}, we get 
\[\vol(R) \geq \frac{1}{2}\vol(R_0) \gg  \frac{(M-16)^k}{k!}\frac{(\LLLL x)^2}{k}. \]
Recall that $M = \log\log x - C \LLLL x$. By using Stirling formula, there exists some positive constants $C_1, C_2 > 0$ such that
\[\begin{split}
\vol(R) & \gg \left(\frac{ e \log 4 \L\L x - C_1\L\L\L\L x}{\L\L x}\right)^{\frac{\L\L x}{\log 4}} \frac{(\L\L\L\L x)^{2}}{(\L\L x)^{3/2}}\\
& \gg (\log x)^{\frac{1+ \log \log 4}{\L 4}} \cdot \frac{(\L\L\L x)^{C_2} (\L\L\L\L x)^{2}}{(\L\L x)^{3/2}},
\end{split} \]
which completes the proof. 
\end{proof}

\begin{dajauthors}

\begin{authorinfo}[maxxu]
  Max Wenqiang Xu\\
  Department of Mathematics, \\
  Stanford University, Stanford, CA 94305, USA\\
  maxxu\imageat{}stanford\imagedot{}edu 
\end{authorinfo}

\begin{authorinfo}[yunkunzhou]
  Yunkun Zhou\\
  Department of Mathematics, \\
  Stanford University, Stanford, CA 94305, USA\\
  yunkunzhou\imageat{}stanford\imagedot{}edu \\
\end{authorinfo}

\end{dajauthors}

\end{document}